\documentclass[12pt]{article}
\usepackage{amsmath, amsthm, amssymb,enumitem}
\usepackage[top=1cm,left=1.5cm,right=1cm,bottom=1.5cm]{geometry}

\theoremstyle{plain}
\newtheorem{theorem}{Theorem}
\newtheorem{lemma}[theorem]{Lemma}

\newtheorem{conjecture}[theorem]{Conjecture}
\newtheorem{corollary}[theorem]{Corollary}

\makeatletter
\g@addto@macro\bfseries{\boldmath}
\makeatother

\usepackage{array,color,colortbl}

\def\eref#1{$(\ref{#1})$}
\def\sref#1{\S$\ref{#1}$}
\def\lref#1{Lemma~$\ref{#1}$}
\def\tref#1{Theorem~$\ref{#1}$}
\def\fref#1{Figure~$\ref{#1}$}
\def\cref#1{Conjecture~$\ref{#1}$}
\def\Cref#1{Corollary~$\ref{#1}$}

\renewcommand{\geq}{\geqslant}
\renewcommand{\leq}{\leqslant}
\renewcommand{\ge}{\geqslant}
\renewcommand{\le}{\leqslant}
\renewcommand{\emptyset}{\varnothing}
\def\Z{\mathbb{Z}}
\def\R{\mathbb{R}}
\def\O{\mathcal{O}}
\def\F{\mathcal{F}}

\def\P{\mathcal{P}}
\def\PP{\mathscr{P}} % Should be different to \P
\def\Y{\mathcal{Y}}
\def\={\equiv}
\def\supp{\text{supp}}
\usepackage{mathrsfs}
\def\ind{\mathscr{S}}
\def\row{\mathscr{R}}
\def\col{\mathscr{C}}

\def\zz{\psi}

\title{Maximal sets of mutually orthogonal frequency squares}

\author{
Nicholas J. Cavenagh\thanks{
Department of Mathematics,
 The University of Waikato,
 Private Bag 3105,
 Hamilton 3240, New Zealand.}\\
\texttt{nickc@waikato.ac.nz}\\ 
\and
Adam Mammoliti
\thanks{
School of Mathematics, Monash University, Clayton 3800, Australia.}\\ 
\texttt{adam.mammoliti@monash.edu}\\ 
\and
Ian M. Wanless
\footnotemark[3]\\ 
\texttt{ian.wanless@monash.edu}\\
}

\begin{document}

\date{}
\maketitle

\begin{abstract}
A \emph{frequency square} is a square matrix in which each row and column is
a permutation of the same multiset of symbols.  A frequency square is
of \emph{type $(n;\lambda)$} if it contains $n/\lambda$ symbols,
each of which occurs $\lambda$
times per row and $\lambda$ times per column. In the case when
$\lambda=n/2$ we refer to the frequency square as \emph{binary}.  A 
set of \hbox{$k$-MOFS$(n;\lambda)$}
is a set of $k$ frequency squares of type $(n;\lambda)$ such that when
any two of the frequency squares are superimposed, each possible
ordered pair occurs equally often.

A set of $k$-maxMOFS$(n;\lambda)$ is a set of $k$-MOFS$(n;\lambda)$
that is not contained in any set of $(k+1)$-MOFS$(n;\lambda)$.  For
even $n$, let $\mu(n)$ be the smallest $k$ such that there exists a
set of $k$-maxMOFS$(n;n/2)$.  It was shown in \cite{BCMW} that
$\mu(n)=1$ if $n/2$ is odd and $\mu(n)>1$ if $n/2$ is even.  Extending
this result, we show that if $n/2$ is even, then $\mu(n)>2$.  Also, we
show that whenever $n$ is divisible by a particular function of $k$,
there does not exist a set of $k'$-maxMOFS$(n;n/2)$ for any $k'\le k$.
In particular, this means that $\limsup \mu(n)$ is unbounded.
Nevertheless we can construct infinite families of maximal binary MOFS
of fixed cardinality.
More generally, let $q=p^u$ be a prime power and let $p^v$ be the
highest power of $p$ that divides $n$. If $0\le v-uh<u/2$ for $h\ge1$ then we
show that there exists a set of $(q^h-1)^2/(q-1)$-{\rm maxMOFS}$(n;n/q)$.

\medskip


\end{abstract}

\section{Introduction}\label{s:intro}

Let $N(n)=\{0,1,\dots,n-1\}$.  In what follows, rows and columns of an
$m\times n$ array $L$ are indexed by $N(m)$ and $N(n)$, respectively,
and $L[i,j]$ denotes the entry in row $i$ and column $j$ of $L$.
A \emph{frequency square} $F$ of \emph{type}
$(n;\lambda_0,\lambda_1,\dots,\lambda_{m-1})$ is an $n\times n$ array
such that symbol $i$ occurs $\lambda_i$ times in each row and
$\lambda_i$ times in each column for each $i\in N(m)$; necessarily
$\sum_{i=0}^{m-1}\lambda_i=n$.  In the case where
$\lambda_0=\lambda_1=\dots=\lambda_{m-1}=\lambda$ we say that $F$ is of
type $(n;\lambda)$. If $\lambda=n/2$, then we refer to the
frequency square as {\em binary}.  A frequency square of type $(n;1)$
is a \emph{Latin square} of order~$n$.  Two frequency squares of type
$(n;\lambda_0,\lambda_1,\dots,\lambda_{m-1})$ are \emph{orthogonal} if
each ordered pair $(i,j)$ occurs $\lambda_i\lambda_j$ times when the
squares are superimposed.  A set of \emph{mutually orthogonal
  frequency squares} (MOFS) is a set of frequency squares in which
each pair of squares is orthogonal.  We use the notation
$k$-MOFS$(n;\lambda)$ to denote $k$ MOFS, each of type
$(n;\lambda)$.

Research into frequency squares focuses mainly on constructions of
sets of MOFS, motivated originally by problems in statistical experiment
design.  Hedayat, Raghavarao and Seiden~\cite{HRS} showed that the
maximum $k$ such that a set of $k$-MOFS$(n;n/m)$ exists is $k=(n-1)^2/(m-1)$;
such a set is called \emph{complete}.  We give a new explanation for
this result (\Cref{cor:newbie}). In the case when $m=n/\lambda>2$,
complete sets of MOFS of type $(n;\lambda)$ are only known to exist
when $n$ is a prime power~\cite{LM,LZD,Mav,street79}; a unified theory
for all known constructions is given in~\cite{JMM}.

A non-prime power result for $m=2$ is given by Federer~\cite{Fe} (see
also~\cite{street79}), who showed that if there exists a
Hadamard matrix of order $n$, then there exists a complete set of MOFS
of type $(n;n/2)$. Conversely, it is shown in \cite{BCMW} that there
does not exist a complete set of MOFS of type $(n;n/2)$ whenever
$n/2$ is odd.

Two sets of frequency squares are \emph{isomorphic} if one can be
obtained from the other by some sequence of the following operations:
\begin{itemize}[topsep=4pt,partopsep=0pt,itemsep=2pt,parsep=2pt]
  \item Applying the same permutation to the rows of all squares in the set.
  \item Applying the same permutation to the columns of all squares in the set.
  \item Transposing all squares in the set.
  \item Permuting the symbols in one of the squares.
  \item Permuting the squares within the set (in cases where we have imposed an
    order on the set).
\end{itemize}
Isomorphism is an equivalence relation and the equivalence classes it
induces are \emph{isomorphism classes}.

A set $\{F_1,F_2,\dots,F_k\}$ of $k$-MOFS$(n;\lambda)$ is said to be
\emph{maximal} if there does not exist a frequency square $F$ of type
$(n;\lambda)$ that is orthogonal to $F_i$ for each $1\leq i\leq k$.
If we wish to specify that a set of $k$-MOFS$(n;\lambda)$ is maximal,
then we may write $k$-maxMOFS$(n;\lambda)$.  Maintaining consistency
with Latin square terminology, a $1$-maxMOFS$(n;\lambda)$ is called a
\emph{bachelor} frequency square.

In \cite{BCMW} a number of existence and non-existence results are
given for sets of $k$-maxMOFS$(n;n/2)$, which we summarise in the next
two theorems.  Note that a set of $(n-1)^2$-MOFS$(n;n/2)$ is complete
and thus trivially maximal.

\begin{theorem}\label{t:rez1} 
There exists a set of $k$-{\rm maxMOFS}$(n;n/2)$ if:
\begin{itemize}
\item $k=1$ and $n\equiv2\pmod 4$ (furthermore such a frequency square
  is unique up to isomorphism);
\item $n=6$ and either $5\leq k\leq 15$ or $k=17$; or 
\item $k=5$, $n\equiv2\pmod 4$ and $n>2$.
\end{itemize}
\end{theorem}

\begin{theorem}\label{t:rez2} 
There does not exist a set of $k$-{\rm maxMOFS}$(n;n/2)$ if:
\begin{itemize}
\item $k=1$ and $n\equiv0\pmod 4$;
\item $n=4$ and $k<9$;  or
\item $n=6$ and either $k\in\{2,3,4,16\}$ or $k\ge18$.
\end{itemize}
\end{theorem}

The structure of this paper is as follows.  In \sref{s:maxMOFS}, we
give a construction for non-complete sets of maxMOFS$(n;n/2)$
when $n\equiv0\pmod 4$ (we are not aware of any earlier construction of this
nature). For instance, we construct $(2^v-1)^2$-maxMOFS$(2^{v}c;2^{v-1}c)$
for all $v\ge2$ and odd $c$. This is part of a more general construction
for sets of maxMOFS$(n;\lambda)$ given by \Cref{cy:maxhighpow2}.  In
\sref{s:asympnonexist} and \sref{s:polytope}, by exploiting the theory
of integral convex polytopes, we show that for any given $k$ there are
infinitely many values of $n$ such that a set of $k'$-maxMOFS$(n;n/2)$
does not exist for any $k'\le k$.  

Finally, in \sref{s:nomono}, we show that a set of
$2$-maxMOFS$(n;n/2)$ does not exist whenever $n$ is divisible by $4$.
The case when $n/2$ is odd remains elusive, but we conjecture the
following (which holds for $n<8$ by \tref{t:rez2}):

\begin{conjecture}
If $n$ is even, then there does not exist a set of $2$-{\rm maxMOFS}$(n;n/2)$. 
\end{conjecture}

For even $n$, define $\mu(n)$ to be the size of the smallest maximal set
of binary MOFS of order $n$.  Then from the above observations,
$\mu(n)=1$ if $n\equiv2\pmod4$, and $2<\mu(n)\le9$ if $n\equiv4\pmod8$,
but $\limsup \mu(n)$ is unbounded.
By comparison, it seems that maximal pairs of orthogonal Latin squares
exist for all orders $n>6$, and this has been proved 
for all orders that are not twice a prime \cite{DWW11}.

\section{A new construction for maximal MOFS}\label{s:maxMOFS}

Here we give a new construction for maximal sets of MOFS via dilations
of complete sets of MOFS.  In the following, we consider $n\times n$
matrices as vectors in an $n^2$-dimensional real vector space equipped
with the inner product $A\circ B=\sum_i\sum_ja_{ij}b_{ij}$ for
matrices $A=[a_{ij}]$ and $B=[b_{ij}]$.  Let $J_n$ be the $n\times n$
matrix with each entry equal to $1$.

\begin{theorem}\label{t:linindep}
Suppose that $\F=\{F_1,F_2,\dots,F_k\}$ is a set of
$k$-{\rm MOFS}$(n;n/m)$ with symbols $N(m)$.
For $r,c \in N(n)$,  $s \in N(m)$ and $1\le t\le k$ we define
\begin{align*}
\row_r[i,j]&=\begin{cases}
1&\text{if }i=r\text{ and}\\
0&\text{otherwise},
\end{cases}\\
\col_c[i,j]&=\begin{cases}
1&\text{if }j=c,\text{ and}\\
0&\text{otherwise},
\end{cases}\\
\ind_{s,t}[i,j]&=\begin{cases}
1&\text{if }F_t[i,j]=s,\text{ and}\\
0&\text{otherwise}.
\end{cases}
\end{align*}
Then $\{J_n\}\cup\{\row_r:1\le r\le n-1\}\cup\{\col_c:1\le c\le n-1\}
\cup\{\ind_{s,t}:1\le s\le m-1,\ 1\le t\le k\}$ is a linearly independent
set.
\end{theorem}

\begin{proof}
Define $\row'_r=n\row_r-J_n$, $\col'_c=n\col_c-J_n$ and
$\ind'_{s,t}=m\ind_{s,t}-J_n$ for each $r,c,s,t$.
It is trivial that $\{J_n\}\cup\{\row_r:1\le r\le n-1\}$ and
$\{J_n\}\cup\{\col_c:1\le c\le n-1\}$ are both linearly independent sets,
as is $\{J_n\}\cup\{\ind_{s,t}:1\le s\le m-1\}$ 
for any given $t$. It follows that 
$\{J_n\}\cup\{\row'_r:1\le r\le n-1\}$,
$\{J_n\}\cup\{\col'_c:1\le c\le n-1\}$ and
$\{J_n\}\cup\{\ind'_{s,t}:1\le s\le m-1\}$ are each linearly independent sets. 
It is also easy to check that 
$$\row'_r\circ\col'_c=\row'_r\circ\ind'_{s,t}=\ind'_{s,t}\circ\col'_c
=\ind'_{s,t}\circ\ind'_{s',t'}=0$$
for all $r,c,s,t,s',t',$ provided $t\ne t'$. Hence
$$\{J_n\}\cup\{\row'_r:1\le r\le n-1\}\cup\{\col'_c:1\le c\le n-1\}
\cup\{\ind'_{s,t}:1\le s\le m-1,\ 1\le t\le k\}$$ 
is a linearly independent set, from which the result follows.
\end{proof}

\begin{corollary}\label{cor:newbie}
If $\F$ is a set of $k$-{\rm MOFS}$(n;n/m)$,
then $k\le(n-1)^2/(m-1)$.
\end{corollary}

\begin{proof}
\tref{t:linindep} exhibited a set of $1+2(n-1)+(m-1)k$ independent
vectors in a $n^2$-dimensional vector space. It follows that
$(m-1)k\le(n-1)^2$.
\end{proof}
\noindent
Note that previous proofs of \Cref{cor:newbie} have been given in
\cite{HRS}, \cite{JP20} and \cite{JMM}. 

Let $\F$ be a set of $k$-MOFS$(n;n/m)$. The
\emph{$d$-dilation} of $\F$ is the set of
$k$-MOFS$(dn;dn/m)$ obtained by replacing every entry $e$ in every
square in $\F$ by a $d\times d$ block of entries, each equal
to $e$. It is trivial to check that $d$-dilation does indeed create a
new set of MOFS. We now explore the more interesting question of
whether $d$-dilation preserves maximality.  We first give some
necessary conditions.

\begin{lemma}\label{l:dilnec}
Let $\F$ be a set of $k$-{\rm MOFS}$(n;n/m)$. Let $\F'$ be the
$d$-dilation of $\F$ and suppose that $\F'$ is maximal. Then $\F$ is
maximal and $d\not\equiv0\pmod m$.
\end{lemma}

\begin{proof}
If $\F$ was not maximal then we could extend it with a new frequency square $F$.
But then the $d$-dilation of $F$ would be orthogonal to every square in $\F'$.
So we can be sure that $\F$ is maximal.

If $d$ is divisible by $m$ then there exists a frequency square $C$ of
type $(d;d/m)$. For example, we could create $C$ as a circulant
matrix in which its first row contains every symbol $d/m$ times. Now
build a frequency square of type $(dn;dn/m)$ orthogonal to every
square in $\F'$, by simply putting a copy of $C$ in the position of
each dilated block. This contradicts the maximality of $\F'$, and
completes the proof.
\end{proof}

Next we give some sufficient conditions. We first need to explain the
idea of a {\em relation}, which was a fundamental tool to show the
existence of maximal sets of MOFS in \cite{BCMW} and \cite{JP20}.  The
technique of relations was previously used in \cite{Do} and \cite{DH}
(with origins in \cite{St}) to analyse maximal sets of mutually
orthogonal Latin squares.

A set $\F=\{F_1,\dots,F_k\}$ of $k$-MOFS$(n;n/m)$ can be written as an
$n^2\times(k+2)$ orthogonal array $\O$ in which there is a row
\begin{equation}\label{e:rowofOA}
\big[i,j,F_1[i,j],F_2[i,j],\dots,F_k[i,j]\big]\,,
\end{equation}
for each $i \in N(n)$ and $j \in N(n)$. In this context it is safest
to consider sets of MOFS to have an indexing that implies an ordering
on the squares (and hence the order of the columns in $\O$ is
well-defined).  Let $Y_c$ be the set of symbols that occur in column
$c$ of $\O$.  Then a \emph{relation} is a $(k+2)$-tuple
$(X_0,\dots,X_{k+1})$ of sets such that $X_i\subseteq Y_i$ for $i \in N(k+2)$, with the property that every row~\eref{e:rowofOA} of $\O$
has an even number of columns $c$ for which the symbol in column $c$
is an element of $X_c$.  
We will consider a particular type of relation from \cite{JP20}, for which we make the following definition.
A \emph{Jedwab-Popatia relation} is a relation such that
$|X_i| = 1$ for $i \geq 2$ and at least one of $\emptyset \subsetneq X_0 \subsetneq Y_0$ and $\emptyset \subsetneq X_1 \subsetneq Y_1$ holds.

The following theorem is proved in \cite{JP20}, generalising an
earlier result from \cite{BCMW}. It shows that under certain
conditions a Jedwab-Popatia relation implies maximality. See \cite{BCMW,JP20}
for a more extensive study of the structure of relations for sets of
MOFS, including restrictions on which relations can be achieved.

\begin{theorem}\label{relationz}
Suppose $\lambda$ is odd and let $\F$ be a set of
$k$-{\rm MOFS}$(n;\lambda)$ that satisfies a Jedwab-Popatia relation. Then $\F$
is maximal.
\end{theorem}

We can now present conditions which guarantee that the $d$-dilation of
a set of MOFS is maximal.

\begin{theorem}\label{t:dilsuff}
Suppose that $\F'$ is the $d$-dilation of a set $\F$ of $k$-{\rm MOFS}$(n;n/m)$.
Then $\F'$ is maximal if either
\begin{itemize}
\item $d^2\not\equiv0\pmod m$ and $\F$ is a complete set of {\rm MOFS} or 
\item $d$ and $n/m$ are odd, and $\F$ satisfies Jedwab-Popatia relation.
\end{itemize} 
\end{theorem}

\begin{proof}
If both $n/m$ and $d$ are odd then $dn/m$ is also odd. Also,
if $\F$ satisfies a Jedwab-Popatia relation then it is easy to see
that $\F'$ also satisfies a Jedwab-Popatia relation, and hence
is maximal by \tref{relationz}.

Hence for the remainder of the proof we may assume that
$\F=\{F_1,\dots,F_k\}$ is a complete set; in other words $k=(n-1)^2/(m-1)$. 

Aiming for a contradiction, assume that there exists a frequency
square $F$ of type $(dn;dn/m)$ that is orthogonal to every square
in the $d$-dilation $\F'=\{ F'_1,\dots,F'_k\}$ of $\F$.  Let
$X=[x_{ij}]$ be the integer matrix of order $n$ in which 
$x_{ij}$ is the number of times that symbol 0 occurs in the $ij$-th block of
$F$.  Since $0$ occurs $dn/m$ times in every row and column of $F$ we have
\begin{align}
X\circ\row_r&=d^2n/m,\label{e:rowsum}\\
X\circ\col_c&=d^2n/m,\label{e:colsum}\\
X\circ J_n&=d^2n^2/m,\label{e:Jsum}
\end{align}
for $r,c \in N(n)$. Also, the fact that $F$ and $F'_t$ 
are orthogonal frequency squares means that
\begin{align}
X\circ\ind_{s,t}&=(dn/m)^2\label{e:XorthM}
\end{align}
for $s \in N(m)$ and $1\le t\le k$.

However \eref{e:rowsum}, \eref{e:colsum}, \eref{e:Jsum} and \eref{e:XorthM}
have another simultaneous solution, namely the matrix with all entries equal to
$d^2/m$. If $d^2\not\equiv0\pmod m$ then this solution is not
an integer matrix, and hence is different from the solution exhibited above.
Having two distinct solutions contradicts \tref{t:linindep}, so
we conclude that $\F'$ must be maximal.
\end{proof}

\begin{corollary}\label{cy:maxhighpow2}
  Let $q=p^u$ be a prime power and let $p^v$ be the highest power of $p$ that
  divides $n$. If $0\le v-uh<u/2$ for $h\ge1$ then there exists a
  set of $(q^h-1)^2/(q-1)$-{\rm maxMOFS}$(n;n/q)$.
\end{corollary}

\begin{proof}
From \cite{HRS} we know that a complete set $\F$ of MOFS$(q^h;q^{h-1})$
exists. Let $d=n/q^h$ and note that $d$ is an integer because $v\ge uh$.
Also, by assumption, the highest power of $p$ dividing
$d$ is $p^{v-uh}$. It follows that $d^2\not\equiv0\pmod q$ since $2(v-uh)<u$.
Hence, applying \tref{t:dilsuff} gives the result.
\end{proof}

Applying \Cref{cy:maxhighpow2} to the binary case we have $q=2$, which
necessitates $v=h$, and we find that there exists a
set of $(2^v-1)^2$-{\rm maxMOFS}$(n;n/2)$ whenever $n\equiv2^v\pmod{2^{v+1}}$.
The $v=1$ case of this statement is just the existence of bachelor frequency squares,
as given in \tref{t:rez1}, but for $v>1$ we get something new.

Next we give an example that shows that a $d$-dilation of a maximal
set of binary MOFS need not itself be maximal, even if $d$ is
odd. This shows that a tempting generalisation of \tref{t:dilsuff}
fails.

Consider the example (16) given in \cite{BCMW} of a set $\F$ of
5-maxMOFS$(6;3)$ that do not satisfy a relation.  Let $\F'$ be the
3-dilation of $\F$. We claim that $\F'$ is not maximal. Indeed, a
frequency square that extends $\F'$ can be obtained from the following
matrix:
\[
\left[
\begin{array}{cccccc}
0&3&2&1&3&0\\
1&0&2&2&3&1\\
3&1&0&2&1&2\\
1&2&2&1&0&3\\
1&1&2&2&0&3\\
3&2&1&1&2&0\\
\end{array}\right].
\]
We simply replace each entry $c$ in this matrix by a frequency square of type $(3;3-c,c)$; for example, a binary $3\times3$ circulant block that has $c$ positive entries in each row.

\section{Asymptotic non-existence results}\label{s:asympnonexist}

In this section we show that $\mu(n)$, the size of the smallest maximal set of
MOFS$(n,n/2)$, does not satisfy any bound that is uniform in $n$.
Instead, we find that for any $k$ there exist infinitely many $n$ for
which $\mu(n)>k$.
This provides an interesting counterpoint to \tref{t:rez1}
and \Cref{cy:maxhighpow2}, both of which provide infinite families of
maximal sets of binary MOFS of a fixed cardinality. For example,
\Cref{cy:maxhighpow2} shows that $\mu(n)\le9$
for all $n\equiv4\pmod8$, and $\mu(n)\le49$
for all $n\equiv8\pmod{16}$, and so on.

Let $\Gamma(m)$ be the least common multiple of the integers
$1,2,\dots,m$.  Let $m_1=m_2=1$, $m_3=2$ and recursively define 
$m_{i+1}=2m_i(m_i-1)\Gamma(2m_i-1)$ for $i\geq 3$.
We will show that:

\begin{theorem}\label{towerofpower}
  If $4m_{2k}$ divides $n$, then there does not exist a set of
  $k$-{\rm maxMOFS}$(n;n/2)$.
\end{theorem}

Note that $m_i$ divides $m_{2k}$ for all $i<2k$, so
\tref{towerofpower} implies that $\mu(n)>k$ if $4m_{2k}$ divides $n$.
We actually prove a more general result which implies
\tref{towerofpower} but does not require the frequency squares to be
orthogonal:

\begin{theorem}\label{towerofpower2}
  If $4m_{2k}$ divides $n$, then given any set
  $\F=\{F_1,F_2,\dots,F_k\}$ of binary frequency squares of order $n$,
  there exists a binary frequency square $F$ which is orthogonal to every
  frequency square in $\F$.
\end{theorem}

Since $\log(\Gamma(m))\sim m$ (see, for example, \cite{RS}),
the sequence $\{m_i\}$
grows asymptotically faster than the tetration (iterated
exponentiation) of $k$ base $e$; so certainly $n\gg k$.

In the remainder of the paper, given arrays $L_1$, $L_2,\dots,L_k$ of
the same dimensions, the {\em join} $L_1\oplus L_2\oplus \dots\oplus L_k$
is defined to be the array obtained by overlapping these arrays;
that is, the array in which cell $(r,c)$ contains the ordered
$k$-tuple $(L_1[r,c],L_2[r,c],\dots,L_k[r,c])$.  Also, given rows
$r_1$ and $r_2$ of any rectangular array $L$, we use $L(r_1,r_2)$ to denote
the two-rowed array in which the first row is equal to row $r_1$ of $L$
and the second row is equal to row $r_2$ of $L$.
We start with an elementary lemma that gives a strategy for constructing
orthogonal mates for frequency squares, two rows at a time.

\begin{lemma}\label{peazy}
Let $F$ be a frequency square of type $(n;n/2)$ and let ${\mathcal R}$
be a partition of the rows of $F$ into pairs.  Suppose there
exists a binary $n\times n$ array $F'$ such that for each
$\{r_1,r_2\}\in {\mathcal R}$:
\begin{itemize}
\item each row of $F'(r_1,r_2)$ contains $n/2$ zeros and $n/2$ ones;  
\item each column of $F'(r_1,r_2)$ contains $1$ zero and $1$ one;
\item within $F(r_1,r_2)\oplus F'(r_1,r_2)$, each of the pairs
  $(0,0)$, $(0,1)$, $(1,0)$ and $(1,1)$ occurs $n/2$ times.  
\end{itemize}
Then $F$ and $F'$ are orthogonal frequency squares of type $(n;n/2)$. 
\end{lemma}

Henceforth in this section, $F$ is the join of a set $\{F_1,F_2,\dots,F_k\}$
of frequency squares, each of type $(n;n/2)$. Let $r_1$, $r_2$ be two
rows of $F$ and let $\P$ be an equipartition of the columns
of $F(r_1,r_2)$.  We say that $\P$ is {\em good} with
respect to row $r_i$ and square $F_j$ (where $i\in \{1,2\}$ and
$1\leq j\leq k$) if:
(a) $|\P|$ is even and each element of $\P$ has
even cardinality; and 
(b) for each $P\in \P$, the
number of columns $c$ in $P$ with $F_j[r_i,c]=1$ is equal to $|P|/2$. 
Note that (b) implies that for each $P\in \P$, the
number of columns $c$ in $P$ with $F_j[r_i,c]=0$ is also equal to $|P|/2$. 

\begin{lemma}\label{twoatatime}
Let $F$ be the join of a set $\F=\{F_1,F_2,\dots,F_k\}$
of frequency squares, each of type $(n;n/2)$ 
and let ${\mathcal R}$ be a partition of the rows of $F$ into
pairs. Suppose that for each $\{r_1,r_2\}\in {\mathcal R}$, there
exists an equipartition $\P$ of the columns of $F(r_1,r_2)$
such that $\P$ is {\em good} with respect to row $r_i$ and
square $F_j$ for each $i\in \{1,2\}$ and $1\leq j\leq k$.  Then there
exists a binary frequency square $F'$ orthogonal to each frequency square in 
$\F$. 
\end{lemma}

\begin{proof}
We construct $F'$ two rows at a time. Let $\{r_1,r_2\}\in {\mathcal R}$
and let $\P$ be an equipartition of the columns of $F(r_1,r_2)$ satisfying
the conditions of the lemma.  Let $\P=\P_0\cup
\P_1$ be any partition of $\P$ into two parts of
equal size.  For each $P\in\P_i$ and $i\in \{0,1\}$, place
$i$ in cell $F'[r_1,c]$ and $1-i$ in cell $F'[r_2,c]$ for
each $c\in P$.  Repeat this process for each element of ${\mathcal R}$.
The result follows from \lref{peazy}.
\end{proof}

For the rest of this section, ${\mathcal R}$ is a partition of the
rows of $F$ into pairs and $\{r_1,r_2\}\in {\mathcal R}$.
\lref{twoatatime} allows us to focus on the array $F(r_1,r_2)$.  The
next lemma is a straightforward observation.

\begin{lemma}\label{sivvable}
Let $\P$ be good with respect to row $r_i$, $i\in \{1,2\}$,
and square $F_j$, $1\leq j\leq k$, for the array $F(r_1,r_2)$.  If
$\P'$ is an equipartition coarser than $\P$ and
$|\P'|$ is even, then $\P$ is also good with
respect to row $r_i$ and square $F_j$.
\end{lemma}

Let $f(1) = 2$ and $f(m)=(2m-2)\Gamma(2m-1)$ for $m \geq 2$.
Informally, the following lemma states that if $\beta$ is divisible by $f(m)$, 
then we can partition any
integer partition of $2m\beta$ with maximum part size $2m$ and average
part size $m$ into $2\beta/f(m)$ integer partitions of $mf(m)$, with
average part size $m$ in each of the smaller partitions.  For the purposes
of motivation, we will apply \lref{problem1} before proving it in the
next section, using the theory of integral convex polytopes.

\begin{lemma}\label{problem1}
Let $m\geq 1$ and $\beta$ be integers such that
$f(m)$ divides $\beta$. Then for any non-negative
integers $x_0,x_1,\dots,x_{2m}$ such that
\begin{equation}\label{e:condforporb}
\sum_{i=0}^{2m} ix_{i} =2m\beta;\quad \sum_{i=0}^{2m} x_i=2\beta 
\end{equation}
there exist non-negative integers $x_{i,j}$, for $0\leq i\leq 2m$ and 
$1\leq j\leq 2\beta/f(m)$ such that 
\begin{equation}\label{e:condforprobsoln}
\sum_{i=0}^{2m} ix_{i,j}= mf(m);\quad    \sum_{i=0}^{2m} x_{i,j}=f(m)
\end{equation}
for each  $1\leq j\leq 2\beta/f(m)$ 
and 
$$\sum_{j=1}^{2\beta/f(m)} x_{i,j}=x_i$$
for each $0\leq i\leq 2m$. 
\end{lemma}

\begin{proof}[Proof of \tref{towerofpower2}]
Consider the array $F(r_1,r_2)$. For $2\leq s\leq 2k$,
define $\beta_s=n/4m_s$.  Our proof will be by induction on $s$.  We first
construct an equipartition $\P$ that is good with respect to
{\em both} row $r_1$ and square $F_1$ {\em and} row $r_2$ and square
$F_1$. Here each $P\in \P$ is a pair of columns $\{c,c'\}$
such that $F_1[r_1,c]=1-F_1[r_1,c']$ and $F_1[r_2,c]=1-F_1[r_2,c']$.
The fact that $F_1$ is a frequency square ensures such a $\P$ exists
(cf. \lref{ezy} later).  Note also that $|\P|=n/2$, which by
assumption is even.

For the inductive step, assume that there exists an equipartition $\P$ which is
good with respect to row $r_i$ and square $F_j$ for a set $S$ of order
pairs $(i,j)\in \{1,2\}\times \{1,2,\dots,k\}$ such that $|S|=s$,
where $2 \leq s< 2k$ and $|\P|=2\beta_s$. 
The base case $s=2$ follows from the previous paragraph, since
$\beta_2=n/4$.

Next, let $(t,u)$ be a fixed pair in
$(\{1,2\}\times\{1,2,\dots,k\})\setminus S$. 
We will show that there exists an equipartition $\P'$ such that: (a)
$\P'$ is coarser than $\P$; (b) $|\P'|=2\beta_{s+1}$ and (c) $\P'$ is
good with respect to row $r_t$ and square $F_{u}$.  The result then
follows by induction and Lemmas~\ref{twoatatime} and~\ref{sivvable}.

Let $\beta=\beta_s$, $m=m_s$ and $\Y=\{1,2,\dots,{2\beta}\}$.  Let
$\P=\{P_j:j\in\Y\}$ be the equipartition of the columns
and for $j\in\Y$ let $y_j$ be the number of $1$'s in row $r_t$ and square
$F_{u}$ within the columns of $P_j$. Then
\begin{equation}\label{e:ys}
\sum_{j\in\Y}y_j=2m\beta
\end{equation}
and 
$0\leq y_j\leq 2m$ for $j\in\Y$.

Next, define $x_i$ to be the number of indices $j\in\Y$ such that $y_j=i$.
Thus by \eref{e:ys}:
$$\sum_{i=0}^{2m} ix_i = 2m\beta \quad\text{and}\quad \sum_{i=0}^{2m} x_i=2\beta.$$ 
By definition, 
$$\beta_{s+1}=\frac{n}{4m_{s+1}}=\frac{n}{4m f(m)}
=\frac{\beta}{f(m)},$$ 
since $m_{\ell+1} = m_\ell f(m_\ell) $ for all $\ell \geq 2$. 
Thus by \lref{problem1}, there exist 
$x_{i,j}$, for $0\leq i\leq 2m$ and $1\leq j\leq 2\beta_{s+1}$
 such that 
\begin{equation}\label{e:part}
  \sum_{i=0}^{2m} ix_{i,j}= mf(m)=m_{s+1}; \quad\sum_{i=0}^{2m} x_{i,j}=f(m)
\end{equation}
for $1\leq j\leq 2\beta_{s+1}$ and 
$$\sum_{j=1}^{2\beta_{s+1}} x_{i,j}=x_i$$
for $0\leq i\leq 2m$.

We now use this information to construct the coarser partition
$\P'$. We do this by partitioning $\Y$ into subsets $Y_j$
for $1\leq j\leq 2\beta_{s+1}$, such that
$Y_j$ contains $x_{i,j}$ indices $\ell$ such that $y_\ell=i$
for each $0\leq i\leq 2m$.  Then, for each $1\leq j\leq 2\beta_{s+1}$, let 
$$P_j':=\bigcup_{\gamma\in Y_j} P_{\gamma}.$$ From \eref{e:part}, each
$P_j'$ is the union of $f(m)$ parts of the equipartition $\P$ which
between them contain $m_{s+1} = n/4\beta_{s+1}$ ones in row $r_t$ of
square $F_{u}$.  Hence $\P'=\{P_j'\;:\; 1\leq j\leq 2\beta_{s+1}\}$ is
an equipartition of the columns which is coarser than $\P$ and is good
with respect to row $r_t$ and square $F_u$.  This completes the proof.
\end{proof}

\section{\label{s:polytope}Proof of \lref{problem1} via integral convex polytopes}

In this section we prove \lref{problem1}.  We first rephrase the
problem in terms of convex polytopes and then use a handy result on the
integer decomposition property of integral convex polytopes from \cite{CHHH}.

We begin with the following definitions.
A halfspace in $\R^n$ is a set of the form 
$\{ \mathbf{x} \in \R^n:\, \mathbf{a}\cdot \mathbf{x} \leq  b \}$ or 
$\{ \mathbf{x} \in \R^n:\, \mathbf{a}\cdot \mathbf{x} \geq  b \}$,
for a fixed vector $\mathbf{a} \in \R^n$ and real number $b$.
A convex polytope $\PP$ is an intersection of halfspaces that is bounded.
The dimension of a convex polytope $\PP \subseteq \R^n$ is the affine
dimension of $\PP$; that is, the smallest dimension $d$ such that a
translation of $\PP$ is contained in a $d$-dimensional subspace of $\R^n$.

We will be interested in the following convex polytopes.
Let $m$ be a fixed positive integer. 
For $b \in \R^{+}$, let $\PP(b)$ be the $(2m-1)$-dimensional convex polytope 
\[
\PP(b) = \left\lbrace(x_0,\dots,x_{2m})\in\R^{2m+1}:\,
\sum_{i=0}^{2m}x_i = b,\, \sum_{i=0}^{2m}ix_i = mb,\, x_i \geq 0 \text{ for } i =0,\ldots, 2m\right\rbrace.
\]

An element $v$ of a convex polytope $\PP$ is a {\em vertex} if the only way to
write $v = ax+(1-a)y$ for $a \in (0,1)$ and $x,y \in \PP$ is to put $x=v=y$.  
Let $\mathbf{e}_\ell\in\R^{2m+1}$ be the $\ell$-th standard basis vector, where
the entries are indexed from $0$. Define
$V=V(b)=\{\mathbf{v}_{i,j}:\,0 \leq i <m <j\leq 2m \} \cup \{\mathbf{v}_{m}\}$, 
where
$\mathbf{v}_{i,j} = \frac{j-m}{j-i}b \,\mathbf{e}_i+\frac{m-i}{j-i}b \,\mathbf{e}_j$
for each $0 \leq i < m < j \leq 2m$ and $\mathbf{v}_m = b\,\mathbf{e}_m$.
It is easy to check that $V\subseteq\PP(b)$.  We show
that in fact $V$ is the set of vertices of $\PP(b)$.

\begin{lemma}\label{l:vert}
The convex polytope $\PP(b)$ has vertex set $V$. 
\end{lemma}

\begin{proof}
First we show that the elements of $V$ are indeed vertices of $\PP(b)$.
All elements of $\PP(b)$ have non-negative entries, so $v = ax +(1-a)y$
with $a \in (0,1)$ and $v,x,y \in \PP(b)$ only if $\supp(x) \subseteq
\supp(v)$ and $\supp(y) \subseteq \supp(v)$.  It is easy to check that
$\mathbf{v}_m$ is the only element of $\PP(b)$ with exactly one
non-zero entry, and hence $\mathbf{v}_m$ is a vertex. Also, $\mathbf{v}_{i,j}$
with $i < m < j$ is a vertex because there is no element $v\in\PP$
distinct from $\mathbf{v}_{i,j}$ with $\supp(v) = \{i,j\}$.

Now we show that no element $x=(x_0,\ldots, x_{2m}) \in \PP(b)\setminus V$ 
is a vertex.
By the previous paragraph, $x$ has at least two non-zero
entries.  If $x_i = 0 $ for all $i<m$, then we have the contradiction
\[
\sum_{i=0}^{2m} i x_i = \sum_{i=m}^{2m} i x_i > m \sum_{i=m}^{2m} x_i = mb,
\]
since $\sum_{i=0}^{2m} x_i = b$ and $x_j > 0$ for at least one $j >m$.
Similarly, we cannot have that $x_i =0$ for all $i > m$.  In
particular, any element of $\PP(b)$ with exactly two non-zero entries
must be one of the $\mathbf{v}_{i,j}$.

So, assume that $x$ has at least 3 non-zero entries. 
By the previous argument,
there must be $i < m < j$ such that $x_i$ and $x_j$ are non-zero.
Thus, we can find $a\in(0,1)$ that is sufficiently small to ensure
that $ab\frac{j-m}{j-i} < x_i$ and 
$ab\frac{m-i}{j-i} < x_j$. 
Then we can write $x = a\mathbf{v}_{i,j}+(1-a)x'$ where
$x' = (x_0',\ldots, x_{2m}')$ with 
$x_i' = \frac{1}{1-a}(x_i -ab\frac{j-m}{j-i})$,
$x_j'=\frac{1}{1-a}(x_j - ab\frac{m-i}{j-i})$ and 
$x_\ell' = \frac{1}{1-a}x_\ell$ if $\ell\notin\{i,j\}$.
By the choice of $a$, all entries of $x'$ are non-negative. 
Moreover, $x' \in \PP$, since 
\begin{align*}
\sum_{\ell = 0}^{2m}x_\ell'
&=\frac{1}{1-a}\left(x_i -ab\frac{j-m}{j-i}\right)+ \frac{1}{1-a}\left( x_j - ab\frac{m-i}{j-i}\right)+\sum_{\ell \neq i,j}\frac{x_\ell}{1-a} \\
&=\frac{1}{1-a}\sum_{\ell = 0}^{2m}x_\ell -\frac{ab}{1-a}
= \frac{b}{1-a}- \frac{ab}{1-a} = b
\end{align*}
and
\begin{align*}
\sum_{\ell = 0}^{2m}\ell x_\ell'
&= \frac{i}{1-a}\left(x_i -ab\frac{j-m}{j-i}\right)+ \frac{j}{1-a}\left(x_j - ab\frac{m-i}{j-i}\right) + \sum_{\ell \neq i,j}\frac{\ell x_\ell}{1-a} \\
&=\frac{1}{1-a}\sum_{\ell = 0}^{2m}\ell x_\ell - \frac{abm}{1-a}
=  \frac{bm}{1-a}- \frac{abm}{1-a} = bm.
\end{align*}
Therefore, $x = a\mathbf{v}_{i,j}+(1-a)x' $ with $0 <a < 1$ and
$\mathbf{v}_{i,j},x' \in \PP$. Moreover, $\mathbf{v}_{i,j}\neq x'$
since $x$ has at least 3 non-zero entries. Thus, $x$
is not a vertex and $V$ is the vertex set of $\PP(b)$, as claimed.
\end{proof}

For a convex polytope $\PP$ and an integer $k$, let
$k\PP = \{k\alpha:\, \alpha\in \PP \}$.
A convex polytope $\PP$ is {\em integral} if every vertex of $\PP$
has integer coordinates.  We remind the reader that $\Gamma(m) =
\text{lcm}(1,\ldots,m)$. The following is immediate from \lref{l:vert}
and the fact that $c\PP(b) = \PP(bc)$ for any integers $b,c$.

\begin{corollary}\label{cor:probisintpoly}
For $m \geq 1$ if $\Gamma(2m-1)$ divides $b$, then the set $\PP(b)$ is
an integral $(2m-1)$-dimensional convex polytope.
\end{corollary}

A convex polytope $\PP \subseteq \R^n$ has the {\em integer
  decomposition property} if for all $k \geq 1$, and $\alpha \in
k\PP \cap \Z^{n}$, there is a way to write
$\alpha=\sum_{i=1}^k \alpha_i $ for some $\alpha_i \in \PP \cap \Z^n$
(such a convex polytope is also called integrally closed).  Note that if a
convex polytope $\PP$ has the integer decomposition property, then so does
$k\PP$ for any integer $k \geq 1$.  The following result is an
immediate consequence of Theorem~1.1 in \cite{CHHH}.

\begin{theorem}\label{thm:intpolyhasidp}
Let $\PP$ be an integral convex polytope of dimension $d \geq 2$. Then  
$(d-1)\PP$ has the integer decomposition property. 
\end{theorem}

We can now prove \lref{problem1}.

\begin{proof}[Proof of \lref{problem1}]
Recall that $f(1) = 2$ and $f(m)=(2m-2)\Gamma(2m-1)$ for $m \geq 2$.
Let $\beta$ be an integer such that $f(m) \mid \beta$ and
$x_0,\ldots,x_{2m}$ be non-negative integers satisfying
\eref{e:condforporb}.  By assumption, $\mathbf{x}=(x_0,\ldots, x_{2m})
\in \PP(2\beta) \cap \Z^{2m+1}$.  By \Cref{cor:probisintpoly} and
\tref{thm:intpolyhasidp}, $\PP(f(m))$ is an integral convex polytope with the
integer decomposition property, when $m>1$.  When $m=1$, observe that
$\PP(2k)= \{(a,c,a)\in \R^{3}\;:\; 2a+c=2k \text{ and } a,c \geq 0 \}$
for any $k\geq 1$ and that the vertex set of $\PP(2)$ is
$\{\mathbf{e}_0+\mathbf{e}_2,2\mathbf{e}_1\}$.  Therefore, for any
$\mathbf{y} \in \PP(2k) \cap \Z^3$,
$\mathbf{y}=a(\mathbf{e}_0+\mathbf{e}_2)+ (c/2)(2\mathbf{e}_1)$ for some
non-negative integers $a,c$ such that $2a+c=2k$ (which in particular
means that $c$ must be even). Thus, $\PP(2)$ has the integer
decomposition property, when $m=1$.  Therefore for any $m \geq 1$,
there exists
$\mathbf{x}_1,\ldots,\mathbf{x}_{2\beta/f(m)}\in\PP(f(m))\cap \Z^{2m+1}$ such that
$\mathbf{x} = \sum_{j=1}^{2\beta/f(m)} \mathbf{x}_j$.
Let $x_{i,j}$ be the $i$-th entry of
$\mathbf{x}_j$, where we index from $0$.  We show that the $x_{i,j}$
satisfy the conclusion of \lref{problem1}.  As
$\mathbf{x}_j\in\PP(f(m))\cap \Z^{2m+1}$ for each $j=1,\ldots, 2\beta/f(m)$,
$x_{0,j},\ldots, x_{2m,j}$
are non-negative integers that satisfy \eref{e:condforprobsoln}.  
The final statement in the conclusion of \lref{problem1}
is immediate from $\mathbf{x} = \sum_{j=1}^{2\beta /f(m)}\mathbf{x}_j$.
\end{proof}

\section{A non-existence result for maximal orthogonal pairs}\label{s:nomono}

\tref{towerofpower} shows that there does not exist a maximal orthogonal pair
of binary frequency squares (that is, a set of $2$-maxMOFS$(2m;m)$)
if $m$ is divisible by $48$.  We improve this significantly in this section
by proving the following:

\begin{theorem}\label{t:main}
If $m$ is even, then there does not exist a set of $2$-{\rm maxMOFS}$(2m;m)$. 
\end{theorem}

For the remainder of the paper, $F_1$ and $F_2$ are binary frequency
squares of order $n=2m$.  Initially we do \emph{not} assume that $F_1$
and $F_2$ are orthogonal. It is plausible that in some application one
might need a frequency square that is orthogonal to each member of a
set of frequency squares, even though the members of that set are not
themselves orthogonal. This viewpoint does materially change what is
possible. For example, below are two superimposed triples of frequency
squares, one of type $(4;2)$, and the other of type $(6;3)$:
\begin{equation}\label{e:nonextend}
\left[
\begin{array}{cccc}
111&011&100&000\\
101&000&010&111\\
010&100&111&001\\
000&111&001&110\\
\end{array}
\right]
\qquad
\left[
\begin{array}{cccccc}
111&111&111&000&000&000\\
111&111&111&000&000&000\\
110&110&000&111&001&001\\
001&001&100&011&110&110\\
000&000&011&100&111&111\\
000&000&000&111&111&111\\
\end{array}
\right]
\end{equation}
Both of these triples are \emph{non-extendable} in the sense that there is no
frequency square of the same type that is orthogonal to all squares in the
triple. This contrasts with \tref{t:rez2}
which showed the non-existence of $3$-maxMOFS$(n;n/2)$ for $n<8$.
Of course, any set of frequency squares of type $(6;3)$ that contains a
bachelor square will be non-extendable. However, a computation
shows that the examples in \eref{e:nonextend} are the smallest non-extendable
sets of order $n\in\{4,6\}$ that do not contain a bachelor square. 

It will be convenient for us to assume that $m$ is even from now on,
although some of our statements apply also to the case when $m$ is
odd.  Since \tref{t:rez2} has completely settled the case $n=4$, we
will assume for the remainder of the paper that
\begin{equation}\label{e:assumptions}
8\le n\equiv0\pmod4,\quad x=\left\lfloor n/6\right\rfloor\geq 1, \hbox{ \ and \ \ }
y=\left\lfloor n/8\right\rfloor\geq 1.
\end{equation}

For a set of rows $S$ of a frequency
square $F$, we define $F(S)$ to be $F$ restricted to the rows in $S$.
When $S=\{r_1,r_2\}$, $F(S)$ is (equivalent to) $F(r_1,r_2)$ defined
in \sref{s:asympnonexist}.  We say that two binary arrays $L$ and $L'$
of the same dimensions are {\em orthogonal} if each of the ordered
pairs $(1,1)$, $(0,0)$, $(0,1)$ and $(1,0)$ occur the same number of
times in $L\oplus L'$, where $\oplus$ was defined in
\sref{s:asympnonexist}.
A {\em binary frequency rectangle} is any matrix of
$0$'s and $1$'s with the same number of $0$'s and $1$'s in each row
and in each column.  For an even subset $S$ of the rows of $F_1\oplus F_2$, we
say that $S$ is {\em balanceable} if there is an $|S| \times n$ binary
frequency rectangle $F$ that is orthogonal to $F_1(S)$ and $F_2(S)$.
Clearly, the union of disjoint balanceable sets is balanceable, so the
following generalisation of \lref{peazy} is immediate.

\begin{lemma}\label{lem:partofbalsetssuff}
 If there exists a partition of the rows of $F_1\oplus F_2$ into
 balanceable sets, then there is a frequency square $F$ orthogonal to
 $F_1$ and $F_2$.
\end{lemma}

We prove \tref{t:main} by finding a suitable partition $\mathcal{R}$ of
the rows of $F_1\oplus F_2$ into balanceable sets and applying
\lref{lem:partofbalsetssuff}.  To do this, we define tools to
analyse pairs of rows in \sref{subsec:pairsofrows}.  We
use these tools to describe all possible pairs of rows that do not
balance and classify them into several different types.  In
\sref{subsec:setsofrows}, we show that it is not possible to
have large sets of rows of a given type that pairwise do not balance.
Finally in \sref{subsec:proofofthm}, we prove \tref{t:main},
using the results of the first four subsections.

\subsection{Preliminaries}

In this subsection we define much of the notation and terminology that
will be needed later in the proof of \tref{t:main}, as well as giving
preliminary results involving those concepts. A detailed example using
these definitions and results can be found in \sref{ss:eg}.

Define $\zz(r)$ to be the number of cells in row $r$ of $F_1\oplus F_2$
which contain $(0,0)$.  Also, given two rows $r_1,r_2$ of a
frequency square $F$, let $\eta(r_1,r_2)$ be the number of columns in
$F$ containing $0$ in row $r_1$ and $r_2$.  The following lemma
is immediate from the definition of a binary frequency square.

\begin{lemma}\label{ezy}
Let $F$ and $F'$ be two binary frequency squares of the same order
$n$. Then in row $r$ of $F\oplus F'$, the number of cells
containing $(1,1)$ is $\zz(r)$, the number of cells containing $(0,1)$
is $m-\zz(r)$ and the number of cells containing $(1,0)$ is
$m-\zz(r)$.  In any rows $r_1$ and $r_2$ of $F$, the number of columns
containing $1$ in both row $r_1$ and row $r_2$ is $\eta(r_1,r_2)$, the
number of columns containing $0$ in row $r_1$ and $1$ in row $r_2$ is
$m-\eta(r_1,r_2)$ and the number of columns containing $1$ in row
$r_1$ and $0$ in row $r_2$ is $m-\eta(r_1,r_2)$.
\end{lemma}

For integers $p$ and $q$ (which may be
negative), we say that a pair of rows $\{r_1,r_2\}$ in $F_1\oplus F_2$
is {\em $(p,q)$-balanceable} if there exists a $2\times n$ binary
frequency rectangle $F$ such that:
\begin{itemize}
\item $F_1(r_1,r_2)\oplus F$ has $m+p$ occurrences of $(0,0)$  {\em and} 
\item $F_2(r_1,r_2)\oplus F$ has $m+q$ occurrences of $(0,0)$.
\end{itemize}
By \lref{ezy}, a pair of rows $\{r_1,r_2\}$ is {\em balanceable}, if and
only if it is $(0,0)$-balanceable; for such a pair of rows, $F$ is
orthogonal to both $F_1(r_1,r_2)$ and $F_2(r_1,r_2)$.

In the above, if we swap the symbols $0$ and $1$ in $F$, then by
\lref{ezy}, $F_1(r_1,r_2)\oplus F$ and $F_2(r_1,r_2)\oplus F$ have
$m-p$ and $m-q$ occurrences of $(0,0)$, respectively.  Thus a pair of
rows is $(p,q)$-balanceable if and only if it is
$(-p,-q)$-balanceable.  The following is immediate.

\begin{lemma}\label{lem:balsets}
Let $S$ be a $2(s+t)$-set of rows of $F_1 \oplus F_2$.  Let $\mathcal{R}$
be a partition of $S$ into pairs $\{r_i,r_i'\}$ such that
$\{r_i,r_i'\}$ are $(p_{i},q_{i})$-balanceable for integers $p_i,q_i$
for $i=1,\ldots, s+t$.  If $\sum_{i=1}^s p_i = \sum_{i=s+1}^t p_i$ and  
$\sum_{i=1}^s q_i= 
\sum_{i=s+1}^t q_i$, 
then $S$ is balanceable.
\end{lemma}

To analyse a pair of rows $\{r_1,r_2\}$ in $F_1 \oplus F_2$,
we use the following definitions.  
Let $[v_i]_{i=1}^4=[(0,1),(1,0)$, $(0,0),(1,1)]$.
Define a $4\times 4$ matrix $A'=A'(r_1,r_2)=[a'_{ij}]$ by letting
$a_{ij}'$ equal the number of columns of $F_1 \oplus F_2$
in which $v_i$ occurs in the first row and $v_j$ occurs in the second row.

\lref{ezy} implies that the sum of the entries in the first row of
$A'$ equals the sum of the entries in the second row of $A'$.
Similarly, the sum of the entries in the third row of $A'$ equals
the sum of the entries in the fourth row of $A'$. Analogous properties
hold for the columns of $A'$. From $A'$ we can also
determine the number of cells containing $(0,0)$ within rows
$r_1$ and $r_2$ of $F_1\oplus F_2$. We summarise these observations in the lemma, below.

\begin{lemma}\label{lem:easyprops4by4}
Let $r_1$ and $r_2$ be two rows in $F_1 \oplus F_2$ and $A'=A'(r_1,r_2)$. 
Then, 
\begin{itemize}
\item the sum of the entries of $A'$ is $n=2m$;
\item $a_{11}'+a_{21}'+a_{31}'+a_{41}'=a_{12}'+a_{22}'+a_{32}'+a_{42}'$; 
\item $a_{11}'+a_{12}'+a_{13}'+a_{14}'=a_{21}'+a_{22}'+a_{23}'+a_{24}'$; 
\item $a_{13}'+a_{23}'+a_{33}'+a_{43}'=a_{14}'+a_{24}'+a_{34}'+a_{44}'$;
\item $a_{31}'+a_{32}'+a_{33}'+a_{34}'=a_{41}'+a_{42}'+a_{43}' +a_{44}'$;
\item  $\zz(r_1)=a_{11}'+a_{13}'+a_{31}'+a_{33}'= a_{22}'+a_{24}'+a_{42}'+a_{44}'$;
\item  $\zz(r_2)=a_{22}'+a_{23}'+a_{32}'+a_{33}'= a_{11}'+a_{14}'+a_{41}'+a_{44}'$.
\end{itemize}
\end{lemma}
\noindent
We say that a $4\times 4$ matrix is {\em admissible} if it satisfies
the above equalities except possibly for the first dot point.
We will sometimes write $A'=B+C$, where $B$ and $C$ are both
admissible matrices.

Swapping the symbols in $F_1$ (respectively, $F_2$) corresponds to
applying the permutation $(12)(34)$ to the rows (respectively,
columns) of $A'$.  Swapping row $r_1$ with $r_2$ corresponds to
applying the permutation $(12)$ to both the rows and the columns of
$A'$. Finally, swapping $F_1$ with $F_2$ corresponds to taking the
transpose of $A'$.  We consider two admissible matrices $A_1'$ and
$A_2'$ to be \emph{equivalent} if $A_2'$ can be formed from $A_1'$ by
some combination of the above operations.

Given that each matrix $A'$ may be equivalent to up to $16$ matrices
satisfying \lref{lem:easyprops4by4}, we often consider a condensed
form of $A'(r_1,r_2)$ which we denote by $A(r_1,r_2)$.  Given a pair
of rows $\{r_1,r_2\}$ in $F_1 \oplus F_2$, we define a $3\times 3$
matrix $A(r_1,r_2)=[a_{ij}]$ as follows.
If $i,j\in\{1,2\}$, then $a_{ij}=a_{ij}'$. 
For $1 \leq i \leq 2$, $a_{i3}=a_{i3}'+a_{i4}'$ and 
for $1\leq j\leq 2$, $a_{3j}=a_{3j}'+a_{4j}'$. 
Finally, $a_{33}=a_{33}'+a_{34}'+a_{43}'+a_{44}'$. 
Informally, $A$ is formed from $A'$ be merging the last two rows
and the last two columns. We have opted to give the simpler notation
to this condensed format because we will use it much
more often than the $4\times4$ version.
We consider the $3 \times 3$ matrices $A_1$ and $A_2$
\emph{equivalent} if $A_2$ can be
formed from $A_1$ by some combination of swapping the first two rows,
swapping the first two columns and/or taking the transpose.
The next lemma is implied by \lref{lem:easyprops4by4}.

\begin{lemma}\label{sumz}
Let $r_1$ and $r_2$ be two rows in $F_1 \oplus F_2$ and $A=A(r_1,r_2)$. 
Then, 
\begin{itemize}
\item the sum of the entries of $A$ is $2m$;
\item $a_{11}+a_{12}+a_{13}=a_{21}+a_{22}+a_{23}=m - (a_{31}+a_{32}+a_{33})/2$; 
\item $a_{11}+a_{21}+a_{31}=a_{12}+a_{22}+a_{32}=m-(a_{13}+a_{23}+a_{33})/2$; 
\item $a_{31}+a_{32}+a_{33}\equiv a_{13}+a_{23}+a_{33}\equiv 0\pmod 2$.  
\end{itemize}
\end{lemma}

We can determine if a pair of rows $\{r_1,r_2\}$ is
$(p,q)$-balanceable by considering only the condensed matrix
$A(r_1,r_2)$, as the following lemma shows.

\begin{lemma}\label{lem:balanz}
Let $r_1,r_2$ be rows in $F_1 \oplus F_2$ and $A=A(r_1,r_2)=[a_{ij}]$.
Suppose there exists a $3\times 3$ matrix $B=[b_{ij}]$ such that 
\begin{itemize}
\item The sum of the entries of $B$ is $m$;
\item $b_{11}+b_{12}+b_{13}-(b_{21}+b_{22}+b_{23})=p$; 
\item $b_{11}+b_{21}+b_{31}-(b_{12}+b_{22}+b_{32})=q$; 
\item $0 \leq b_{ij}\leq a_{ij}$ for $1\leq i\leq 3$ and $1\leq j\leq 3$.  
\end{itemize}
Then the pair of rows $\{r_1,r_2\}$ is $(p,q)$-balanceable in $F_1 \oplus F_2$. 
\end{lemma}

\begin{proof}
For $1\leq i\leq 3$ and $1\leq j\leq 3$, partition
those columns of $F_1 \oplus F_2$ that are counted by $a_{ij}$
into sets $C_{ij}$ and $C_{ij}'$ of cardinalities $b_{ij}$ and
$a_{ij}-b_{ij}$, respectively. Such a partition exists, since
$0 \leq b_{ij}\leq a_{ij}$. 
It follows that $\{C_{ij},C_{ij}':1\leq i,j\leq 3\}$
partitions the columns of $F_1 \oplus F_2$.  We construct a $2\times 2m$
binary frequency rectangle $F$ satisfying the properties required for
$\{r_1,r_2\}$ to be $(p,q)$-balanceable, as follows.  For each column
$c$, place a $0$ in the first row and a $1$ in the second row of $F$ 
if $c \in C_{ij}$ for some $i,j$ and place a $1$ in the first row and
a $0$ in the second row of $F$, otherwise.  By \lref{sumz}, the
total number of pairs $(0,0)$ in $F_1(r_1,r_2)\oplus F$ is given by:
$$b_{11}+b_{12}+b_{13}+(a_{21}-b_{21})+(a_{22}-b_{22})+(a_{23}-b_{23})
+(a_{31}+a_{32}+a_{33})/2=m+p.$$ 
Similarly, the
total number of pairs $(0,0)$ in $F_2(r_1,r_2)\oplus F$ is given by:
\begin{equation*}
  b_{11}+b_{21}+b_{31}+(a_{12}-b_{12})+(a_{22}-b_{22})+(a_{32}-b_{32})
  + (a_{13}+a_{23}+a_{33})/2=m+q.\qedhere
\end{equation*}
\end{proof}

We say that $A=A(r_1,r_2)$ is {\em $(p,q)$-balanceable} if
$\{r_1,r_2\}$ is $(p,q)$-balanceable.  A matrix $B$ satisfying the
conditions in \lref{lem:balanz} is said to {\em $(p,q)$-balance} $A$ and
$\{r_1,r_2\}$. When $B$ $(0,0)$-balances $A$, we just say that $B$
{\em balances} $A$ and $\{r_1,r_2\}$.  

If $B$ is such that it $(p,q)$-balances $A$, then taking the
transpose, swapping the first two rows or swapping the first two
columns of $A$ and $B$, results in matrices $A'$ and $B'$,
respectively, such that $B'$ $(q,p)$-balances, $(-p,q)$-balances or
$(p,-q)$-balances $A'$, respectively.  We therefore sometimes only
need to consider $(p,q)$-balanceability up to equivalence.  In
particular, a matrix $A$ can be balanced if and only if any matrix
equivalent to $A$ can be balanced.

\subsection{\label{ss:eg}A detailed example}

We give four rows of $F_1\oplus F_2$ where $F_1$ and $F_2$ are each of
order $8$:
$$\begin{array}{r|c|c|c|c|c|c|c|c|}
\cline{2-9}
r_1 & (0,0) & (0,0) & (0,0) & (1,0) & (1,1) & (1,1) & (0,1) & (1,1) \\
\cline{2-9}
r_2 & (1,1) & (1,1) & (1,0) & (0,0) & (0,1) & (0,1) & (0,0) & (1,0) \\
\cline{2-9}
r_3 & (0,0) & (0,1) & (1,0) & (1,1) & (0,0) & (1,1) & (0,1) & (1,0) \\
\cline{2-9}
r_4 & (1,1) & (1,0) & (0,1) & (0,0) & (1,0) & (0,1) & (0,0) & (1,1) \\
\cline{2-9}
\end{array}$$

Then: 
$$
A(r_1,r_2)=\begin{array}{|c|c|c|}
\hline
2 & 0 & 1 \\
\hline
0 & 0 & 3 \\
\hline
0 & 2 & 0 \\
\hline
\end{array}
\quad 
A'(r_1,r_2)=\begin{array}{|c|c|c|c|}
\hline
2 & 0 & 1 & 0 \\
\hline
0 & 0 & 1 & 2  \\
\hline
0 & 1 & 0 & 0 \\
\hline
0 & 1 & 0 & 0 \\
\hline
\end{array}
\quad 
A(r_3,r_4)=\begin{array}{|c|c|c|}
\hline
1 & 1 & 1 \\
\hline
1 & 1 & 1 \\
\hline
1 & 1 & 0 \\
\hline
\end{array}
\quad 
A'(r_3,r_4)=\begin{array}{|c|c|c|c|}
\hline
1 & 1 & 1 & 0 \\
\hline
1 & 1 & 0 & 1  \\
\hline
0 & 1 & 0 & 0 \\
\hline
1 & 0 & 0 & 0 \\
\hline
\end{array}\,.
$$
Furthermore the matrix
$$B=\begin{array}{|c|c|c|}
\hline
1 & 0 & 1 \\
\hline
0 & 0 & 1 \\
\hline
0 & 1 & 0 \\
\hline
\end{array}$$
$(1,0)$-balances both $A(r_1,r_2)$ and $A(r_3,r_4)$, by
\lref{lem:balanz}.  Note that the pair of rows $\{r_1,r_2\}$ is not
balanceable (this is an instance of exception $E_1$ in
\lref{lem:exceptions}, which we will prove shortly).  However, by
\lref{lem:balsets}, the set of rows $S=\{r_1,r_2,r_3,r_4\}$ is
balanceable.  Indeed we exhibit a $4\times 8$ binary frequency
rectangle $F$ orthogonal to both $F_1(S)$ and $F_2(S)$:
$$
F=\begin{array}{|c|c|c|c|c|c|c|c|}
\hline
0 & 1 & 0 & 0 & 1 & 1 & 0 & 1 \\
\hline
1 & 0 & 1 & 1 & 0 & 0 & 1 & 0 \\
\hline
1 & 0 & 0 & 0 & 1 & 1 & 1 & 0 \\
\hline
0 & 1 & 1 & 1 & 0 & 0 & 0 & 1 \\
\hline
\end{array}\,.$$

\subsection{Pairs of rows that do not balance}\label{subsec:pairsofrows}
We next determine all
matrices that correspond to a pair of rows that is not balanceable.

\begin{lemma}\label{lem:exceptions}
Let $A$ be a matrix with the properties from \lref{sumz}.
Then there exists a matrix $B$ that balances $A$, {\em unless}
$A$ is equivalent to one of the follow configurations $E_i$, $1\leq
i\leq 6$, where $x$ and $y$ are defined in \eref{e:assumptions}.
$$
\begin{array}{|c|c|c|}
\hline
2x & 0  & 1 \\
\hline
  0  & 0  & 2x+1  \\
\hline
  0 & 2x & 0\\
\hline
 \multicolumn{3}{c}{E_1: n\equiv 2 \pmod 6}   
\end{array} 
\quad \quad
\begin{array}{|c|c|c|}
\hline
 2x+1 & 0 &  0  \\
\hline
 0  & 1  & 2x  \\
\hline
  0 &2x & 0 \\
\hline
 \multicolumn{3}{c}{E_2: n\equiv 2 \pmod 6}   
\end{array}
\quad \quad
\begin{array}{|c|c|c|}
\hline
 2x+1  & 0 &  0  \\
\hline
 1  & 0  & 2x  \\
\hline
 0 & 2x+2 & 0 \\
\hline
 \multicolumn{3}{c}{E_3: n\equiv 4 \pmod 6}   
\end{array}
$$
$$
\begin{array}{|c|c|c|}
\hline
 2x+1  & 0 &  0  \\
\hline
 0  &0  & 2x+1 \\
\hline
 0 & 2x+1 & 1 \\
\hline
 \multicolumn{3}{c}{E_4: n\equiv 4 \pmod 6}   
\end{array}
\quad \quad
\begin{array}{|c|c|c|}
\hline
 2y+1  & 0 &  0  \\
\hline
 2y+1  & 0  & 0  \\
\hline
 0 & 4y+2 & 0 \\
\hline
 \multicolumn{3}{c}{E_5: n\equiv 4 \pmod 8}   
\end{array}
\quad \quad
\begin{array}{|c|c|c|}
\hline
 2y+1  & 0 &  0  \\
\hline
 2y  &0  & 1 \\
\hline
 0 & 4y+1 & 1 \\
\hline
 \multicolumn{3}{c}{E_6: n\equiv 4 \pmod 8}   
\end{array}
$$
\end{lemma}

\begin{proof}
We consider cases according to the parity of $a_{11},a_{12},a_{21}$
and $a_{22}$.  For each case, we either present a matrix $B$ that
balances $A$ or show that $A$ must be equivalent to one of the
exceptional configurations in the lemma statement.
Throughout the proof, we make extensive use of \lref{sumz} and, for simplicity, 
we omit referencing the lemma every time it is used.

\noindent{\bf Case 1: $a_{11}$, $a_{12}$, $a_{21}$ and $a_{22}$
  have the same parity.}  

\noindent{\bf Case 1A: $a_{11}\=a_{12}\=a_{21}\=a_{22}\pmod 2$
  and $a_{13}\=a_{31}\pmod 2$.}\
Note that $a_{33}$ is even.
A solution for $B$ in this case is:
\def\arraystretch{1.05}

$$\begin{array}{|c|c|c|}
\hline
 \lfloor a_{11}/2 \rfloor & 
 \lceil a_{12}/2 \rceil & 
 \lceil a_{13}/2 \rceil \\
\hline
 \lceil a_{21}/2 \rceil & 
 \lfloor a_{22}/2 \rfloor & 
 \lceil a_{23}/2 \rceil \\
\hline
 \lfloor a_{31}/2 \rfloor & 
 \lfloor a_{32}/2 \rfloor & 
  a_{33}/2 \\
\hline
\end{array}\;.
$$
So, in all other cases we may assume that $a_{13}\not\=a_{31}\pmod 2$.
By transposing if necessary, we may assume that $a_{13}$ is odd and $a_{31}$ is
even.

\noindent{\bf Case 1B: $a_{11}\=a_{12}\=a_{21}\=a_{22}\pmod 2$,
  $a_{13}$ is odd and $a_{31}$ is even  and $a_{33}>0$.}\
A solution for $B$ in this case is:
$$
\begin{array}{|c|c|c|}
\hline
 \lfloor a_{11}/2 \rfloor & 
 \lceil a_{12}/2 \rceil & 
  (a_{13}-1)/2  \\
\hline
 \lceil a_{21}/2 \rceil & 
 \lfloor a_{22}/2 \rfloor & 
 (a_{23}-1)/2  \\
\hline
  a_{31}/2 & 
  a_{32}/2 & 
 a_{33}/2 +1\\
\hline
\end{array}\;.
$$

\noindent{\bf Case 1C: $a_{11}$, $a_{12}$, $a_{21}$, $a_{22}$ and $a_{13}$
  are odd, $a_{31}$ is even and $a_{33}=0$.}\
As $n \equiv 0 \pmod 4$, at least one of $a_{31}$ or $a_{32}$
is non-zero. A solution is to take $B$ equivalent to 
$$\begin{array}{|c|c|c|}
\hline
(a_{11}-1)/2 & (a_{12}-1)/2  & (a_{13}+1)/2  \\
\hline
  (a_{21}-1)/2  & (a_{22}+1)/2  & (a_{23}-1)/2  \\
\hline
  (a_{31}+2)/2 & a_{32}/2 & 0\\
\hline
\end{array}\;.$$

\noindent{\bf Case 1D: $a_{11}$, $a_{12}$, $a_{21}$, $a_{22}$ and $a_{31}$ are even, $a_{13}$ is odd and $a_{33}=0$.}\ 
If $a_{11}$ and $a_{31}$ are non-zero then a solution for $B$ is 
$$\begin{array}{|c|c|c|}
\hline
(a_{11}-2)/2 & a_{12}/2  & (a_{13}+1)/2  \\
\hline
  a_{21}/2  & a_{22}/2  & (a_{23}-1)/2  \\
\hline
  (a_{31}+2)/2 & a_{32}/2 & 0\\
\hline
\end{array}\;.$$
A similar solution exists if $a_{31} \neq 0$ and $a_{21} \neq 0$ or
$a_{32}$ and one of $a_{12}$ and $a_{22}$ are non-zero.  As
$n \equiv0 \pmod 4$, at least one of $a_{31}$ or $a_{32}$ is non-zero.
Therefore, without loss of generality, it suffices to consider the
following configuration.
$$\begin{array}{|c|c|c|}
\hline
0 & a_{12}  & a_{13} \\
\hline
  0  & a_{22}  & a_{23}  \\
\hline
  a_{31} & 0 & 0\\
\hline
\end{array}\;, 
$$
where either $a_{12}$ or $a_{22}$ is non-zero.
If $a_{12} \neq 0 $ and $a_{13} \neq 1$, then a solution for $B$ is
$$\begin{array}{|c|c|c|}
\hline
0 & (a_{12}+2)/2  & (a_{13}-3)/2 \\
\hline
  0  & a_{22}/2  & (a_{23}-1)/2  \\
\hline
  (a_{31}+2)/2 & 0 & 0\\
\hline
\end{array}\;.
$$
A similar solution exists when $a_{22} \neq 0$ and $a_{23} \neq 1$.
So, without loss of generality,
$a_{12} \neq 0$ and $a_{13} =1$.
If $a_{22} \neq 0$ (and $a_{23}=1$), then we have the following configuration: 
$$\begin{array}{|c|c|c|}
\hline
0 & a_{12}  & 1 \\
\hline
  0  & a_{12}  & 1  \\
\hline
  2a_{12} & 0 & 0\\
\hline
\end{array}\;
$$
which is a contradiction as $n \equiv 0 \pmod 4$.
So, $a_{22}=0$ and we have the exceptional case $E_1$:  
$$\begin{array}{|c|c|c|}
\hline
0 & a_{12}  & 1 \\
\hline
  0  & 0  & a_{12}+1  \\
\hline
  a_{12} & 0 & 0\\
\hline
\end{array}\;. 
$$

\noindent{\bf Case 2: Precisely 3 of $a_{11}$, $a_{12}$, $a_{21}$ and $a_{22}$ have the same parity.} 

Without loss of generality, we can assume that $a_{12},a_{21}$ and
$a_{22}$ have the same parity.  Necessarily $a_{13} \not\equiv a_{23}\pmod 2$
and $a_{31} \not\equiv a_{32} \pmod 2$ and $a_{33} \equiv 1\pmod 2$.
In particular $a_{33} \geq 1$.

\noindent{\bf Case 2A: $a_{11}$ even and $a_{12}$, $a_{21}$ and
  $a_{22}$ are odd.}\
Without loss of generality, we consider the
three cases: $a_{13}$ and $a_{31}$ are both odd; $a_{13}$ is odd and
$a_{31}$ is even; $a_{13}$ and $a_{31}$ are both even.  Solutions
for $B$ in these respective cases are:
$$\begin{array}{|c|c|c|}
\hline
 a_{11}/2  & (a_{12}+1)/2 &  (a_{13}-1)/2  \\
\hline
 (a_{21}+1)/2  & (a_{22}-1)/2  & a_{23}/2  \\
\hline
  (a_{31}-1)/2 & a_{32}/2 & (a_{33}+1)/2 \\
\hline
\end{array} \;,  
  \quad  \quad \quad 
  \begin{array}{|c|c|c|}
\hline
 a_{11}/2  & (a_{12}+1)/2 &  (a_{13}-1)/2  \\
\hline
 (a_{21}+1)/2  & (a_{22}-1)/2  & a_{23}/2  \\
\hline
  a_{31}/2 & (a_{32}+1)/2 & (a_{33}-1)/2 \\
\hline
\end{array}\;,$$
$$
\begin{array}{|c|c|c|}
\hline
 a_{11}/2  & (a_{12}+1)/2 &  a_{13}/2  \\
\hline
 (a_{21}+1)/2  & (a_{22}+1)/2  & (a_{23}-1)/2  \\
\hline
  a_{31}/2 & (a_{32}-1)/2 & (a_{33}-1)/2 \\
\hline
\end{array}\;.
$$

\noindent {\bf Case 2B: $a_{11}$ is odd, $a_{12}$, $a_{21}$ and $a_{22}$ are even
  and at least one of $a_{13}$ and $a_{31}$ is odd.}\

Without loss of generality, we can consider the two cases when
$a_{13}$ and $a_{31}$ are both odd, and when
$a_{13}$ is odd and $a_{31}$ is even.
Solutions in these respective cases are 
$$\begin{array}{|c|c|c|}
\hline
 (a_{11}+1)/2  & a_{12}/2 &  (a_{13}-1)/2  \\
\hline
 a_{21}/2  & a_{22}/2  & a_{23}/2  \\
\hline
  (a_{31}-1)/2 & a_{32}/2 & (a_{33}+1)/2 \\
\hline
\end{array}\;,
  \quad \quad \quad 
\begin{array}{|c|c|c|}
\hline
 (a_{11}+1)/2  & a_{12}/2 &  (a_{13}-1)/2  \\
\hline
 a_{21}/2  & a_{22}/2  & a_{23}/2  \\
\hline
  a_{31}/2 & (a_{32}+1)/2 & (a_{33}-1)/2 \\
\hline
\end{array}\;.$$

\noindent {\bf Case 2C: $a_{11}$ is odd and
  $a_{12}$, $a_{21}$, $a_{22}$, $a_{13}$ and $a_{31}$ are even.}\
For subcases (a)~$a_{33}>1$, (b)~$a_{22}\neq 0$, (c)~$a_{31} \neq 0$,
(d)~$a_{12} \neq 0$, $a_{21} \neq 0$, and
(e)~$a_{12}=0$, $a_{21} \neq 0$, $a_{23} \geq 3$,
respectively, solutions for $B$ are:
$$\begin{array}{|c|c|c|}
\hline
 (a_{11}+1)/2  & a_{12}/2 &  a_{13}/2  \\
\hline
 a_{21}/2  & a_{22}/2  & (a_{23}+1)/2  \\
\hline
  a_{31}/2 & (a_{32}+1)/2 & (a_{33}-3)/2 \\
\hline
\end{array}\;,
\quad \quad \quad
\begin{array}{|c|c|c|}
\hline
 (a_{11}-1)/2  & a_{12}/2 &  a_{13}/2  \\
\hline
 a_{21}/2  & (a_{22}-2)/2  & (a_{23}+1)/2  \\
\hline
  a_{31}/2 & (a_{32}+1)/2 & (a_{33}+1)/2 \\
\hline
\end{array}\;, 
$$
$$
\begin{array}{|c|c|c|}
\hline
 (a_{11}+1)/2  & a_{12}/2 &  a_{13}/2  \\
\hline
 a_{21}/2  & a_{22}/2  & (a_{23}+1)/2  \\
\hline
  (a_{31}-2)/2 & (a_{32}-1)/2 & (a_{33}+1)/2 \\
\hline
\end{array}\;,
  \quad \quad \quad
\begin{array}{|c|c|c|}
\hline
 (a_{11}+1)/2  & (a_{12}-2)/2 &  a_{13}/2  \\
\hline
 (a_{21}-2)/2  & a_{22}/2  & (a_{23}+1)/2  \\
\hline
  a_{31}/2 & (a_{32}+1)/2 & (a_{33}+1)/2 \\
\hline
\end{array}\;,  
$$
$$
\begin{array}{|c|c|c|}
\hline
 (a_{11}-1)/2  & a_{12}/2 &  a_{13}/2  \\
\hline
 (a_{21}+2)/2  & a_{22}/2  & (a_{23}-3)/2  \\
\hline
  a_{31}/2 & (a_{32}+1)/2 & (a_{33}+1)/2 \\
\hline
\end{array}\;.$$
Therefore, without loss of generality, the remaining cases are when $a_{22}=a_{31}=a_{13}=0$ and $a_{33}=1$,
and either $a_{12} =0 =a_{21}$ or $a_{12}=0$ and $a_{23}=1$.  If
$a_{12}=a_{21}=0$ then we have the following exceptional case $E_4$:
\[
\begin{array}{|c|c|c|}
\hline
 a_{11}  & 0 &  0  \\
\hline
 0  &0  & a_{11} \\
\hline
  0 & a_{11} & 1 \\
\hline
\end{array}\;.
\]

If $a_{12} =0$ and $a_{23}=1$, then we have the following exceptional
case $E_6$:
\[
\begin{array}{|c|c|c|}
\hline
 a_{11}  & 0 &  0  \\
\hline
 a_{11}-1  &0  & 1 \\
\hline
  0 & 2a_{11}-1 & 1 \\
\hline
\end{array}\;.
\]

\noindent{\bf Case 3: $a_{11} \equiv a_{22} \pmod 2$, $a_{12} \equiv a_{21} \pmod 2$ and 
$a_{11}\not\equiv a_{12}\pmod 2$.} 

Without loss of generality, we only need to consider the case when
$a_{11}$ and $a_{22}$ are odd and $a_{12}$ and $a_{21}$ are even. We
necessarily have that $a_{13} \equiv a_{23} \pmod 2$ and
$a_{31}\equiv a_{32} \pmod 2$ and $a_{33}$ is even.

\noindent{\bf Case 3A: $a_{11}$, $a_{22}$ and $a_{13}$ are odd,
  while $a_{12}$ and $a_{21}$ are even.}\ 
Solutions when $a_{31}$ is odd (respectively even) are:
$$\begin{array}{|c|c|c|}
\hline
 (a_{11}+1)/2  & a_{12}/2 &  (a_{13}-1)/2  \\
\hline
 a_{21}/2  & (a_{22}-1)/2  & (a_{23}+1)/2  \\
\hline
  (a_{31}-1)/2 & (a_{32}+1)/2 & a_{33}/2 \\
\hline
\end{array}\;,  
\quad  \quad \quad
  \begin{array}{|c|c|c|}
\hline
 (a_{11}+1)/2  & a_{12}/2 &  (a_{13}-1)/2  \\
\hline
 a_{21}/2  & (a_{22}+1)/2  & (a_{23}-1)/2  \\
\hline
  a_{31}/2 & a_{32}/2 & a_{33}/2 \\
\hline
\end{array}\;.
$$

\noindent {\bf Case 3B: $a_{11}$ and $a_{22}$ are odd, while $a_{12}$, $a_{21}$,
  $a_{13}$ and $a_{31}$ are even.}\ 
If $a_{33}\neq 0$ we have this solution
$$\begin{array}{|c|c|c|}
\hline
 (a_{11}+1)/2  & a_{12}/2 &  a_{13}/2  \\
\hline
 a_{21}/2  & (a_{22}+1)/2  & a_{23}/2  \\
\hline
  a_{31}/2 & a_{32}/2 & (a_{33}-2)/2 \\
\hline
\end{array}\;.$$
So, henceforth we assume that $a_{33}=0$.
As $n \equiv 0 \pmod 4$,  at least one of $a_{13}$ and $a_{23}$ is
non-zero and at least one of $a_{31}$ and $a_{32}$ are non-zero.  So,
without loss of generality $a_{31}\neq 0$.  If
$a_{23} \neq 0$, then a solution is
$$\begin{array}{|c|c|c|}
\hline
 (a_{11}+1)/2  & a_{12}/2 &  a_{13}/2  \\
\hline
 a_{21}/2  & (a_{22}-1)/2  & (a_{23}+2)/2  \\
\hline
  (a_{31}-2)/2 & a_{32}/2 &0 \\
\hline
\end{array}\;.$$
So, without loss of generality it remains to consider the case when
$a_{23}=a_{32}=0$ and $a_{13}\neq 0$.  If $a_{11} \geq 3$, then a
solution is
$$\begin{array}{|c|c|c|}
\hline
 (a_{11}-3)/2  & a_{12}/2 &  (a_{13}+2)/2  \\
\hline
 a_{21}/2  & (a_{22}-1)/2  & 0  \\
\hline
  (a_{31}+2)/2 &0 & 0 \\
\hline
\end{array}\;.$$
If $a_{11}=1$ and $a_{21} \neq 0$, then a solution is 
$$\begin{array}{|c|c|c|}
\hline
 (a_{11}-1)/2  & a_{12}/2 &  (a_{13}+2)/2  \\
\hline
 (a_{21}+2)/2  & (a_{22}-1)/2  & 0  \\
\hline
  (a_{31}-2)/2 &0 & 0 \\
\hline
\end{array}\;.$$
So, without loss of generality $a_{12}=0=a_{21}$ and $a_{11}=1$ and
we get the exception $E_2$: 
$$\begin{array}{|c|c|c|}
\hline
 1 & 0 &  a_{13}  \\
\hline
 0  & a_{13}+1  & 0  \\
\hline
  a_{13} &0 & 0 \\
\hline
\end{array}\;.$$

\noindent By equivalence, only the following case remains.

\noindent{\bf Case 4: $a_{11}\equiv a_{21} \pmod 2$, 
$a_{12}\equiv a_{22} \pmod 2$ and $a_{11}\not\equiv a_{12}\pmod 2$.} 

Without loss of generality, we can assume that $a_{11}$ and $a_{21}$
are odd and we necessarily have that $a_{13} \equiv a_{23} \pmod 2$,
$a_{31} \equiv a_{32} \pmod 2$ and $a_{33}$ is even.

\noindent{\bf Case 4A: $a_{11}$, $a_{21}$ and $a_{13}$ are odd, while
  $a_{12}$ and $a_{22}$ are even.}\
Solutions when $a_{31}$ is odd (respectively, even) are
$$\begin{array}{|c|c|c|}
\hline
 (a_{11}+1)/2  & a_{12}/2 &  (a_{13}-1)/2  \\
\hline
 (a_{21}+1)/2  & a_{22}/2  & (a_{23}-1)/2  \\
\hline
  (a_{31}-1)/2 & (a_{32}+1)/2 & a_{33}/2 \\
\hline
\end{array}\;,
\quad \quad \quad 
\begin{array}{|c|c|c|}
\hline
 (a_{11}+1)/2  & a_{12}/2 &  (a_{13}-1)/2  \\
\hline
 (a_{21}-1)/2  & a_{22}/2  & (a_{23}+1)/2  \\
\hline
  a_{31}/2 & a_{32}/2 & a_{33}/2 \\
\hline
\end{array}\;.$$

\noindent{\bf Case 4B: $a_{11}$ and $a_{21}$ are odd, while
  $a_{12}$, $a_{22}$ and $a_{13}$ are even and $a_{33}>0$.}\
If $a_{31}$ is odd then a solution for $B$ is
$$\begin{array}{|c|c|c|}
\hline
 (a_{11}+1)/2  & a_{12}/2 &  a_{13}/2  \\
\hline
 (a_{21}+1)/2  & a_{22}/2  & a_{23}/2  \\
\hline
  (a_{31}-1)/2 & (a_{32}+1)/2 & (a_{33}-2)/2 \\
\hline
\end{array}\;.
$$
The subcase when $a_{13}$ and $a_{31}$ are both even requires a more thorough
analysis.  A solution if $a_{13} \neq 0$ (similarly $a_{23}\neq 0$),
and $a_{31} \neq 0$ are, respectively:
$$
\begin{array}{|c|c|c|}
\hline
 (a_{11}+1)/2  & a_{12}/2 &  (a_{13}-2)/2  \\
\hline
 (a_{21}-1)/2  & a_{22}/2  & a_{23}/2  \\
\hline
  a_{31}/2 & a_{32}/2 & (a_{33}+2)/2 \\  
\hline
\end{array}\;, 
\quad \quad \quad 
\begin{array}{|c|c|c|}
\hline
 (a_{11}+1)/2  & a_{12}/2 &  a_{13}/2  \\
\hline
 (a_{21}+1)/2  & a_{22}/2  & a_{23}/2  \\
\hline
  (a_{31}-2)/2 & a_{32}/2 & a_{33}/2 \\  
\hline
\end{array}\;. 
$$
So, it remains to consider configurations of the form
$$\begin{array}{|c|c|c|}
\hline
 a_{11}  & a_{12} &  0  \\
\hline
 a_{21}  & a_{22}  & 0  \\
\hline
  0 & a_{32} & a_{33} \\
\hline
\end{array}\;.$$
Note that $a_{32}\ne 0$, since otherwise we would have the contradiction
$a_{11}=a_{22}$. If $a_{12} \neq 0$, then the following is a solution
for $B$:
$$\begin{array}{|c|c|c|}
\hline
 (a_{11}+1)/2  & (a_{12}-2)/2 &  0  \\
\hline
 (a_{21}-1)/2  & a_{22}/2  & 0  \\
\hline
  0 & (a_{32}+2)/2 & a_{33}/2 \\
\hline
\end{array}\;.$$
A similar solution exists if $a_{22} \neq 0$. The final case is when
$a_{12}=0=a_{22}$ (necessarily $a_{21}=a_{11}$ and $a_{32}=2a_{11}$).
As $n \equiv 0 \pmod 4$ and $a_{33} \neq 0$, we have $a_{33} \geq 4$ and the
following is a solution:
$$\begin{array}{|c|c|c|}
\hline
 (a_{11}+1)/2  & 0 &  0  \\
\hline
 (a_{21}+1)/2  & 0  & 0  \\
\hline
  0 & (a_{32}+2)/2 & (a_{33}-4)/2 \\
\hline
\end{array}\;.$$

\noindent{\bf Case 4C: $a_{11}$ and $a_{21}$ are odd, while
  $a_{12}$, $a_{22}$ and $a_{13}$ are even and $a_{33}=0$.}\
First consider the subcase when $a_{31}$ is odd.  As
$n\equiv 0 \pmod 4$, at least one of $a_{13}$ or $a_{23}$ is
non-zero. So, without loss of generality, $a_{13} \neq 0$ and then a
solution for $B$ is:
$$\begin{array}{|c|c|c|}
\hline
 (a_{11}+1)/2  & a_{12}/2 &  (a_{13}-2)/2  \\
\hline
 (a_{21}-1)/2  & a_{22}/2  & a_{23}/2  \\
\hline
 (a_{31}+1)/2 & (a_{32}+1)/2 & 0 \\
\hline
\end{array}\;.$$
Now we consider the case when $a_{13}$ and $a_{31}$ are both even.
As $n \equiv 0 \pmod 4$ at least one of 
$a_{31}$ and $a_{32}$ is non-zero.
If $a_{31} \neq 0$ then a solution for $B$ is 
$$\begin{array}{|c|c|c|}
\hline
 (a_{11}+1)/2  & a_{12}/2 &  a_{13}/2  \\
\hline
 (a_{21}+1)/2  & a_{22}/2  & a_{23}/2  \\
\hline
  (a_{31}-2)/2 & a_{32}/2 & 0 \\
\hline
\end{array}\;.$$
So, without loss of generality, $a_{31}=0$ and $a_{32}\neq 0$.
If $a_{12} \neq 0$, then a solution is 
$$\begin{array}{|c|c|c|}
\hline
 (a_{11}-1)/2  & (a_{12}+2)/2 &  a_{13}/2  \\
\hline
 (a_{21}+1)/2  & a_{22}/2  & a_{23}/2  \\
\hline
  0 & (a_{32}-2)/2 & 0 \\
\hline
\end{array}\;.$$
A similar solution exists when $a_{22} \neq 0$.
So, without loss of generality, let $a_{12}=a_{22}=0$. 
If $a_{13}$ and $a_{23}$ are both non-zero, then a solution is
$$\begin{array}{|c|c|c|}
\hline
 (a_{11}+1)/2  & 0 &  (a_{13}-2)/2  \\
\hline
 (a_{21}+1)/2  & 0  & (a_{23}-2)/2  \\
\hline
  0 & (a_{32}+2)/2 & 0 \\
\hline
\end{array}\;.$$
So, without loss of generality, $a_{23}=0$.
If $a_{11} \geq 3$ and $a_{13} \geq 4$, then a solution for $B$ is 
$$\begin{array}{|c|c|c|}
\hline
 (a_{11}+3)/2  & 0 &  (a_{13}-4)/2  \\
\hline
 (a_{21}-1)/2  & 0  & 0  \\
\hline
  0 & (a_{32}+2)/2 & 0 \\
\hline
\end{array}\;.$$
So either $a_{11}=1$ or $a_{13} \leq 2$.
If $a_{11}=1$, then we obtain the following exceptional case $E_3$:  
$$\begin{array}{|c|c|c|}
\hline
 1  & 0 &  a_{13}  \\
\hline
 a_{13}+1  & 0  & 0  \\
\hline
  0 & a_{13}+2 & 0 \\
\hline
\end{array}\;.$$
If $a_{13}=2$ or $a_{13}=0$ then we have following, respectively: 
$$
\begin{array}{|c|c|c|}
\hline
 a_{11}  & 0 &  2  \\
\hline
 a_{11}+2  & 0  & 0  \\
\hline
  0 & 2a_{11}+2 & 0 \\
\hline
\end{array}
\quad \text{or} \quad
\begin{array}{|c|c|c|}
\hline
 a_{11}  & 0 &  0  \\
\hline
 a_{11}  & 0  & 0  \\
\hline
  0 & 2a_{11} & 0 \\
\hline
\end{array}\;.
$$
The former is impossible as $n \equiv 0 \pmod 4$. The latter is the
exception $E_5$.
\end{proof}

Before studying the exceptional configurations of
\lref{lem:exceptions} in more detail, we first note the following
simple consequences of the lemma. We start by noting that
we have shown a special case of \tref{t:main}.

\begin{corollary}\label{cy:mod24}
  If $n \equiv 0 \pmod {24}$, then there is no set of
  $2$-{\rm maxMOFS}$(n;n/2)$.
\end{corollary}

\begin{proof}
\lref{lem:exceptions} gives a complete list of configurations which
correspond to pairs of rows which are not balanceable and none can
occur if $n \equiv 0 \pmod {24}$. The result then follows by
\lref{lem:partofbalsetssuff}.
\end{proof}

Since we have not used the fact that $F_1$ is orthogonal to $F_2$ in
deriving \lref{lem:exceptions}, we have the following, more general
corollary, which improves \tref{towerofpower2} when $k=2$.

\begin{corollary}\label{cy:mod242}  
If $n \equiv 0 \pmod {24}$ and $F_1$ and $F_2$ are any frequency
squares of type $(n;n/2)$, then there exists a frequency square $F$, also
of type $(n;n/2)$, that is orthogonal to both $F_1$ and $F_2$.
\end{corollary}

As all exceptional configurations listed in \lref{lem:exceptions} satisfy
$a_{33}\leq 1$, we have the following.

\begin{corollary}\label{goodcases}
Let $r$ and $r'$ be two distinct rows in $F_{1} \oplus F_2$. If there
exists two distinct columns $c$ and $c'$ such that $F[r,c]=F[r',c]$
and $F[r,c']=F[r',c']$, then the pair $\{r,r'\}$ is balanceable.
\end{corollary}

As mentioned earlier, $A(r_1,r_2)$ is sufficient to determine the
balanceability of the rows $r_1$ and $r_2$.  We say that
$A'=A'(r_1,r_2)$ is an $E_i$ if $A(r_1,r_2)$ is equivalent to $E_i$
for $i \in \{1,\ldots, 6\}$. All such matrices characterise pairs
of rows that do not balance, by \lref{lem:exceptions}. Unpacking that
result, we find that the exceptional cases are each based on
one of two underlying structures:

\begin{lemma}\label{lem:standexc}
Let
\[
D=\begin{array}{|c|c|c|c|}
\hline
2x & 0  & 0 & 0 \\
\hline
  0  & 0  & x & x  \\
\hline
  0 & x & 0 & 0 \\
\hline
0 & x & 0 & 0 \\
\hline
\end{array}
\quad \text{and} \quad 
F= \begin{array}{|c|c|c|c|}
\hline
2y & 0  & 0 & 0 \\
\hline
 2y  & 0  & 0 & 0  \\
\hline
  0 & 2y & 0 & 0 \\
\hline
0 & 2y & 0 & 0 \\
\hline
\end{array}\;.
\] 
Then up to equivalence $E_i=D+B_i$ for $i\in\{1,2,3,4\}$ and $E_i=F+B_i$ for $i\in\{5,6\}$, where
\[
B_1=\begin{array}{|c|c|c|c|}
\hline
0 & 0  & 1 & 0 \\
\hline
  0  & 0  & 0 & 1  \\
\hline
  0 & 0 & 0 & 0 \\
\hline
0 & 0 & 0 & 0 \\
\hline
\end{array}\;,
\quad 
B_2= \begin{array}{|c|c|c|c|}
\hline
1 & 0  & 0 & 0 \\
\hline
0  & 1  & 0 & 0  \\
\hline
  0 & 0 & 0 & 0 \\
\hline
0 & 0 & 0 & 0 \\
\hline
\end{array}\;,
\quad  
B_3=B_5= \begin{array}{|c|c|c|c|}
\hline
1 & 0  & 0 & 0 \\
\hline
1  & 0  & 0 & 0  \\
\hline
  0 & 1 & 0 & 0 \\
\hline
0 & 1 & 0 & 0 \\
\hline
\end{array}
\]
and 
\[
\quad 
B_4=B_6= \begin{array}{|c|c|c|c|}
\hline
1 & 0  & 0 & 0 \\
\hline
0  & 0  & 1 & 0  \\
\hline
  0 & 1 & 0 & 0 \\
\hline
0 & 0 & 0 & 1 \\
\hline
\end{array}
\quad \text{or} \quad 
\begin{array}{|c|c|c|c|}
\hline
1 & 0  & 0 & 0 \\
\hline
0  & 0  & 0 & 1  \\
\hline
  0 & 1 & 0 & 0 \\
\hline
0 & 0 & 1 & 0 \\
\hline
\end{array}\;.
\] 
\end{lemma}

We say that a pair of rows $\{r,r'\}$ and the matrices $A(r,r')$ and
$A'(r,r')$ are of type $\alpha$, $\beta$ or $\gamma$ if, respectively,
$A(r,r')$ is equivalent to an element of $\{E_1,E_2\}$, $\{E_3,E_4\}$ or
$\{E_5,E_6\}$ from \lref{lem:exceptions}.  We say a row
$r$ is of type $\alpha$, $\beta$ or $\gamma$ if there is an $r'$ such
that $\{r,r'\}$ has type $\alpha$, $\beta$ or $\gamma$, respectively.
In the next three lemmas we further categorise rows and pairs of rows
by the number $\zz(r)$ of occurrences of $(0,0)$ in each row~$r$.
Each lemma is a consequence of 
\lref{lem:easyprops4by4} and \lref{lem:standexc}, 
by considering the exceptional configurations and 
equivalences.

\begin{lemma}\label{lem:num00alpha}
Let $\{r_1,r_2\}$ be a pair of rows of type $\alpha$.  Then the
multiset $\{\zz(r_1),\zz(r_2)\}$ is either $\{2x,2x+1\}$,
$\{2x+1,2x+1\}$, $\{x+1,x\}$ or $\{x,x\}$.  In the first two cases we
say that $\{r_1,r_2\}$ has type $\alpha_1$, and we say the pair $\{r_1,r_2\}$
has type $\alpha_2$ otherwise.
\end{lemma}

\begin{lemma}\label{lem:num00beta}
Let $\{r_1,r_2\}$ be a pair of rows of type $\beta$.  Then the
multiset $\{\zz(r_1),\zz(r_2)\}$ is one of $\{2x+1,2x+1\}$,
$\{2x+1,2x+2\}$, $\{x+1,x+1\}$ or $\{x+1,x\}$.  In the first two cases
we say the pair $\{r_1,r_2\}$ has type $\beta_1$ and we say the pair
$\{r_1,r_2\}$ has type $\beta_2$ otherwise.
\end{lemma}

A row $r$ has type $\alpha_1$, $\alpha_2$, $\beta_1$ or $\beta_2$ if it is
in a pair $\{r,r'\}$ that is of type $\alpha_1$, $\alpha_2$, $\beta_1$
or $\beta_2$, respectively.  Note that we do not claim that every row
has exactly one type in the previous two lemmas; we show the contrary
in \lref{l:not2difftype} below.  For technical reasons, we do not
further classify pairs of rows of type $\gamma$ here, only individual
rows, as below.

\begin{lemma}\label{lem:num00gamma}
Let $\{r_1,r_2\}$ be a pair of rows of type $\gamma$.  Then the
multiset $\{\zz(r_1),\zz(r_2)\}$ is one of $\{2y,2y+1\}$,
$\{2y+1,2y+1\}$, $\{2y+1,2y+2\}$.  For $i\in\{1,2\}$, we say that row $r_i$
has type $\gamma_1$ if $\zz(r_i)=2y+1$ and type $\gamma_2$ otherwise.
\end{lemma}

As mentioned above, it is possible for a row to have more than one of
the types $\alpha_i$, $\beta_i$ or $\gamma_i$.  We now show, however, that
in most instances no such row exists.

\begin{lemma}\label{l:not2difftype}
  A row $r$ is of two different types from the set
  $\{\alpha_1,\alpha_2,\beta_1,\beta_2,\gamma_1,\gamma_2\}$ only if
  \begin{itemize}
  \item[(i)] $n=8$, $\zz(r)=2$ and $r$ has type $\alpha_1$ and $\alpha_2$, or
  \item[(ii)] $n=20$, $\zz(r)\in\{4,6\}$ and $r$ has type $\gamma_2$ and either
    $\alpha_1$ or $\alpha_2$.
  \end{itemize}
\end{lemma}

\begin{proof}
Note that by definition, $r$ can never be of type $\gamma_1$ and
$\gamma_2$.  By Lemmas~\ref{lem:num00alpha} and~\ref{lem:num00beta},
$r$ has type $\alpha_1$ and $\alpha_2$ or $\beta_1$ and $\beta_2$ only
if $2x \leq x+1$ or $2x +1 \leq x+1$, respectively. As $x \geq 1$, $r$
is never type $\beta_1$ and $\beta_2$ and has type $\alpha_1$ and
$\alpha_2$ only if $x=1$ and $\zz(r)=2$.  So
we can assume that $r$ has two of the three types $\alpha$, $\beta$ and
$\gamma$.  The row $r$ can never be of type $\alpha$ and $\beta$, since
any type $\alpha$ row can only occur when $n\equiv 2 \pmod 6$ while
any row of type $\beta$ only occurs when $n \equiv 4 \pmod 6$.  By
Lemmas~\ref{lem:exceptions},~\ref{lem:num00beta} 
and~\ref{lem:num00gamma}, $r$ has type $\beta$ and
$\gamma$ only if $n \equiv 4 \pmod{24}$ and at least one of
$2y\le x+1$ or $2y+2\ge2x+1$ holds. However, these inequalities imply
$n/4-1\leq (n+2)/6$ and $n/4+1 \geq (n-1)/3$, respectively, 
both contradicting the fact that $4<n \equiv 4 \pmod{24}$.  Finally,
by similar reasoning to the previous case, $r$ has type $\alpha$ and
$\gamma$ only if $n \equiv 20 \pmod{24}$ and at least one of $2y\le x+1$
or $2y+2\ge2x$ holds.  However, this implies $n/4-1 \leq (n+4)/6$
or $n/4+1 \geq (n-2)/3$; in both cases $n\le20$.  It follows that $r$
has type $\alpha$ and $\gamma$ only if $n=20$, $\zz(r)\in\{4,6\}$ and
$r$ has type $\gamma_2$ and one of $\alpha_1$ or $\alpha_{2}$.
\end{proof}

Observe the following property of types $\alpha$ and $\beta$ (the same does not hold for type $\gamma$).

\begin{lemma}\label{lem:one2two}
If $\{r_1,r_2\}$ is a pair of rows of type $\alpha_i$ (respectively $\beta_i$),
where $i\in \{1,2\}$, then after swapping $0$ and $1$ in exactly one of
the frequency squares $F_1$ and $F_2$, the pair $\{r_1,r_2\}$ is
of type $\alpha_{3-i}$ (respectively, $\beta_{3-i}$).
\end{lemma}

For a matrix $A'=A'(r,r')$, its \emph{dual} is the equivalent matrix formed from
$A'$, by applying the permutation $(1 2)(3 4)$ to its rows and
columns. The following is then a corollary of \lref{lem:standexc}.

\begin{corollary}\label{cor:isstandordualis}
Let $A'=A'(r,r')$ have type $\alpha_1$. Then either $A'$ or its dual is one of the following (up to transpose): 
$$\begin{array}{|c|c|c|c|}
\hline
2x & 0  & 1 & 0 \\
\hline
  0  & 0  & x & x+1  \\
\hline
  0 & x & 0 & 0 \\
\hline
0 & x & 0 & 0 \\
\hline
\end{array}
\quad \text{or} \quad 
\begin{array}{|c|c|c|c|}
\hline
2x & 0  & 0 & 1 \\
\hline
  0  & 0  & x+1 & x  \\
\hline
  0 & x & 0 & 0 \\
\hline
0 & x & 0 & 0 \\
\hline
\end{array}
\quad \text{or} \quad 
\begin{array}{|c|c|c|c|}
\hline
2x+1 & 0  & 0 & 0 \\
\hline
  0  & 1  & x & x  \\
\hline
  0 & x & 0 & 0 \\
\hline
0 & x & 0 & 0 \\
\hline
\end{array}\;.$$
Let $A'=A'(r,r')$ have type $\beta_1$. Then either $A'$ or its dual
is one of the following (up to transpose): 
$$\begin{array}{|c|c|c|c|}
\hline
2x+1 & 0  & 0 & 0 \\
\hline
  1  & 0  & x & x  \\
\hline
  0 & x+1 & 0 & 0 \\
\hline
0 & x+1 & 0 & 0 \\
\hline
\end{array}
\quad \text{or} \quad 
\begin{array}{|c|c|c|c|}
\hline
2x+1 & 0  & 0 & 0 \\
\hline
  0  & 0  & x+1 & x  \\
\hline
  0 & x+1 & 0 & 0 \\
\hline
0 & x & 0 & 1 \\
\hline
\end{array}
\quad \text{or}
$$ $$
\begin{array}{|c|c|c|c|}
\hline
2x+1 & 0  & 0 & 0 \\
\hline
  0  & 0  & x & x+1  \\
\hline
  0 & x & 1 & 0 \\
\hline
0 & x+1 & 0 & 0 \\
\hline
\end{array}
\quad \text{or} \quad 
\begin{array}{|c|c|c|c|}
\hline
2x+1 & 0  & 0 & 0 \\
\hline
  0  & 0  & x & x+1  \\
\hline
  0 & x+1 & 0 & 0 \\
\hline
0 & x & 1 & 0 \\
\hline
\end{array}\;.$$
\end{corollary}

\subsection{Sets of rows that pairwise do not balance}\label{subsec:setsofrows}

We next give upper bounds on the number of rows in $F_1 \oplus F_2$
that are pairwise not balanceable and of the same type. We begin with
rows of type $\alpha$.

\begin{lemma}\label{lem:4alpharows}
If $n > 8$, then any set of four rows of type $\alpha$ in
$F_1 \oplus F_2$ contains a balanceable pair.
\end{lemma}

\begin{proof}
Let $r,r_1,r_2$ and $r_3$ be $4$ rows of type $\alpha$ in $F_1 \oplus F_2$.
We may assume that $\{r,r_i\}$ is not balanceable for
$i\in\{1,2,3\}$, since otherwise we are done.  It follows from
Lemmas~\ref{lem:num00gamma} and~\ref{l:not2difftype} that $\{r,r_i\}$
is of type $\alpha$ for $i\in\{1,2,3\}$, since, even if $n=20$, there
can be at most one row in a non-balanceable pair which is of two
different types.  By \lref{lem:one2two}, we can assume that
$A'(r,r_1)$ is of type $\alpha_1$.  \lref{lem:num00alpha} implies that
$A'(r,r_2)$ and $A'(r,r_3)$ both have type $\alpha_1$.  Let $C_{00}$
be the set of columns of $F_1 \oplus F_2$ for which row $r$ contains a
$(0,0)$.  By \Cref{cor:isstandordualis}, by taking the dual if necessary,
we can assume that $A'(r,r_1)$ and $A'(r,r_2)$ are each one of
\[
\begin{array}{|c|c|c|c|}
\hline
2x & 0  & 0 & 1 \\
\hline
  0  & 0  & x+1 & x  \\
\hline
 0 & x & 0 & 0 \\
\hline
0 & x & 0 & 0 \\
\hline
\end{array}
\quad \text{and} \quad 
\begin{array}{|c|c|c|c|}
\hline
2x & 0  & 0 & 0 \\
\hline
  0  & 0  & x & x  \\
\hline
 0 & x+1 & 0 & 0 \\
\hline
1 & x & 0 & 0 \\
\hline
\end{array}.
\]
when $\zz(r)=2x$ and $A'(r,r_1)$ and $A'(r,r_2)$ are each one of 
\[
\begin{array}{|c|c|c|c|}
\hline
2x+1 & 0  & 0 & 0 \\
\hline
  0  & 1  & x & x  \\
\hline
 0 & x & 0 & 0 \\
\hline
0 & x & 0 & 0 \\
\hline
\end{array}
\quad \text{and} \quad 
\begin{array}{|c|c|c|c|}
\hline
2x & 0  & 1 & 0 \\
\hline
  0  & 0  & x & x+1  \\
\hline
 0 & x & 0 & 0 \\
\hline
0 & x & 0 & 0 \\
\hline
\end{array}
\quad \text{and} \quad 
\begin{array}{|c|c|c|c|}
\hline
2x & 0  & 0 & 0 \\
\hline
  0  & 0  & x & x  \\
\hline
 1 & x & 0 & 0 \\
\hline
0 & x+1 & 0 & 0 \\
\hline
\end{array} 
\]
when $\zz(r)=2x+1$.  In particular, there is
at most one column in $C_{00}$ such that row $r_i$ of $F_1 \oplus F_2$ does
not contain $(1,1)$, for each $i\in\{1,2\}$.  Thus, there are at least $2x-2$
columns in $C_{00}$ such that both rows $r_1$ and $r_2$ of $F_1 \oplus
F_2$ contain a $(1,1)$.  As $x \geq 2$, the result follows from
\Cref{goodcases}.
\end{proof}

The above lemma is also true when $n=8$, using a slightly more
complicated argument, but this will not be necessary.  We show an
analogous result for rows of type $\beta$.

\begin{lemma}\label{lem:4betarows}
Any set of four rows of type $\beta$ in $F_1 \oplus F_2$ contains
a balanceable pair.
\end{lemma}

\begin{proof}
Let $r,r_1,r_2$ and $r_3$ be $4$ rows of type $\beta$ in $F_1 \oplus F_2$.
We may assume that $\{r,r_i\}$ is not balanceable for $i\in\{1,2,3\}$,
since otherwise we are done. By \lref{l:not2difftype}, it follows that 
$\{r,r_i\}$ is of type $\beta$ for $i\in\{1,2,3\}$. 
By \lref{lem:one2two},  we can assume that $A'(r,r_1)$ is of type
$\beta_1$.  \lref{lem:num00beta} and \lref{l:not2difftype}
then imply that $A'(r,r_2)$ and $A'(r,r_3)$ are also of type $\beta_1$. 
Let $C_{00}$ be the set of columns with $(0,0)$ in row $r$ of
$F_1 \oplus F_2$.  By 
\Cref{cor:isstandordualis}, by taking the dual if necessary,
we can assume that $A'(r,r_1)$ and $A'(r,r_2)$ have $2x+1$
in cell $(1,1)$ and $\zz(r)\in\{2x+1,2x+2\}$.
Thus for each $i\in\{1,2\}$, there are at least $2x+1$ columns of
$C_{00}$ for which $r_i$ contains a $(1,1)$ and at most one column of
$C_{00}$ for which $r_i$ does not contain $(1,1)$.  So there are at
least $2x\ge2$ columns of $C_{00}$ that contain a $(1,1)$ in both row
$r_1$ and $r_2$. It follows from \Cref{goodcases} that $\{r_1,r_2\}$
is balanceable.
\end{proof}

We now consider rows of type $\gamma$.
We start by categorising configurations equivalent to $E_5$, as given below.
\[
\begin{array}{|c|c|c|}
\hline
 2y+1  & 0 &  0  \\
\hline
 2y+1  & 0  & 0  \\
\hline
 0 & 4y+2 & 0 \\
\hline
 \multicolumn{3}{c}{}  \\[-13pt]
 \multicolumn{3}{c}{T_1}   
\end{array}
\quad \quad
\begin{array}{|c|c|c|}
\hline
0  &  2y+1 &  0  \\
\hline
 0  &  2y+1  & 0  \\
\hline
 4y+2 & 0 & 0 \\
\hline
 \multicolumn{3}{c}{}  \\[-13pt]
 \multicolumn{3}{c}{T_2}   
\end{array}
\quad \quad
\begin{array}{|c|c|c|}
\hline
 2y+1  & 2y+1 &  0  \\
\hline
 0  & 0  & 4y+2  \\
\hline
 0 & 0 & 0 \\
\hline
 \multicolumn{3}{c}{}  \\[-13pt]
 \multicolumn{3}{c}{T_3}  
\end{array}
\quad \quad
\begin{array}{|c|c|c|}
\hline
 0  & 0 &  4y+2  \\
\hline
 2y+1  & 2y+1  & 0  \\
\hline
 0 & 0 & 0 \\
\hline
 \multicolumn{3}{c}{}  \\[-13pt]
 \multicolumn{3}{c}{T_4}   
\end{array}
\]
\fref{figgy2} depicts each of the types given above as they would
appear in $F_1 \oplus F_2$, up to permuting the columns.

\begin{figure}[h]
$$\begin{array}{r|ccc|ccc|ccc|ccc|}
\multicolumn{1}{l}{} & 
\multicolumn{1}{l}{\longleftarrow} & 
\multicolumn{1}{c}{n/4} & 
\multicolumn{1}{r}{\longrightarrow} &
\multicolumn{1}{l}{\longleftarrow} & 
\multicolumn{1}{c}{n/4} & 
\multicolumn{1}{r}{\longrightarrow} & 
\multicolumn{1}{l}{\longleftarrow} & 
\multicolumn{1}{c}{n/4} & 
\multicolumn{1}{r}{\longrightarrow} & 
\multicolumn{1}{l}{\longleftarrow} & 
\multicolumn{1}{c}{n/4} & 
\multicolumn{1}{r}{\longrightarrow} \\
\cline{2-13}
r & (0,0) & \cdots & (0,0) & (1,1) & \cdots & (1,1) & (1,0) & \cdots & (1,0) & (0,1) & \cdots & (0,1) \\
\cline{2-13}
\cline{2-13}
T_1 & (1,1) & \cdots & (1,1) & (1,0) & \cdots & (1,0) & (0,1) & \cdots & (0,1) & (0,0) & \cdots & (0,0) \\
%\cline{2-13}
T_2 & (0,1) & \cdots & (0,1) & (0,0) & \cdots & (0,0) & (1,1) & \cdots & (1,1) & (1,0) & \cdots & (1,0) \\
%\cline{2-13}
T_3 & (1,1) & \cdots & (1,1) & (0,1) & \cdots & (0,1) & (0,0) & \cdots & (0,0) & (1,0) & \cdots & (1,0) \\
%\cline{2-13}
T_4 & (1,0) & \cdots & (1,0) & (0,0) & \cdots & (0,0) & (0,1) & \cdots & (0,1) & (1,1) & \cdots & (1,1) \\
\cline{2-13}
\end{array}$$
\caption{\label{figgy2}Some type $\gamma$ rows.} 
\end{figure}

Define each $E_6$ as type $T^*_i$ if $T_i$ is obtained from $T^*_i$ by changing
exactly $4$ entries.  We say that $A=A(r_1,r_2)$ has {\em type $T_i$} and
$r_2$ has {\em type $T_i$ with respect to $r_1$} if $A$ forms an $E_5$
of type $T_i$. We define type $T^*_i$ similarly and define
$A'(r_1,r_2)$ to be the same type as $A(r_1,r_2)$ for any pairs of
rows $r_1$ and $r_2$.  

We find several of the pairs $(p,q)$ for which the types defined above
are $(p,q)$-balanceable.  

\begin{lemma}\label{lem:nearbalE5E6}
Let $\{r,r'\}$ be a pair of rows in $F_1 \oplus F_2$ such that
$A=A(r,r')$ is equivalent to $E_5$ or $E_6$. If $A$ is of type $T^*_i$
for some~$i$, then $A$ is $(p,q)$-balanceable for
$(p,q) \in \{(0,1),(1,0),(1,1),(1,-1)\}$. If $A$ is of type $T_1$ or $T_2$, then
$A$ is $(1,0)$-balanceable and if $A$ is of type $T_3$ or $T_4$ then
$A$ is $(0,1)$-balanceable.
\end{lemma}
\begin{proof}
Consider the following array of type $T_1^*$,
\begin{equation}\label{e:T1}
\begin{array}{|c|c|c|}
\hline
 2y+1  & 0 &  0  \\
\hline
 2y  &0  & 1 \\
\hline
 0 & 4y+1 & 1 \\
\hline
\end{array}\;.
\end{equation}
By \lref{lem:balanz},
the following $4$ matrices  
$(p,q)$-balance the array above for $(p,q) =(0,1),(1,0),(1,1)$ and $(1,-1)$, respectively: 
\[
\begin{array}{|c|c|c|}
\hline
 y+1  & 0 &  0  \\
\hline
 y  &0  & 1 \\
\hline
 0 & 2y & 0 \\
\hline
\end{array}
\quad \quad
\begin{array}{|c|c|c|}
\hline
 y+1  & 0 &  0  \\
\hline
 y  &0  & 0 \\
\hline
 0 & 2y+1 & 0 \\
\hline
\end{array}
\quad \quad
\begin{array}{|c|c|c|}
\hline
 y+1  & 0 &  0  \\
\hline
 y  &0  &0 \\
\hline
 0 & 2y & 1 \\
\hline
\end{array}
\quad \quad
\begin{array}{|c|c|c|}
\hline
 y+1  & 0 &  0  \\
\hline
 y-1  &0  & 1 \\
\hline
 0 & 2y+1 & 0 \\
\hline 
\end{array}\;.
\]
By noting that an array is $(p,q)$-balanceable if and only if it is
$(-p,-q)$-balanceable, it follows that the array \eref{e:T1} is
$(p,q)$-balanceable for $(p,q) \in \{-1,0,1\}^2\setminus\{(0,0)\}$.
As the set $\{-1,0,1\}^2\setminus\{(0,0)\}$ is preserved under
negation of an entry or swapping the entries of the ordered pairs, it
follows that any array equivalent to \eref{e:T1}, is also
$(p,q)$-balanceable for $(p,q) \in \{-1,0,1\}^2\setminus\{(0,0)\}$.
As all arrays of type $E_6$ are equivalent, this proves the result
when $A = A(r,r')$ is of type $T_i^*$ with $i \in \{1,2,3,4\}$.  The
matrix above that $(1,0)$-balances \eref{e:T1} also $(1,0)$-balances
$A$ if $A$ is of type $T_1$.   Swapping the first two columns of
$T_1$ results in $T_2$, while $T_3$ and $T_4$ are the transpose of $T_1$
and $T_2$, respectively. It follows that $A$ is $(1,0)$-balanceable if
$A$ is of type $T_2$ and $A$ is $(0,1)$-balanceable if $A$ is of type
$T_3$ or $T_4$.
\end{proof}

\noindent
We also need to consider the balanceability of rows of type $\alpha$.

\begin{lemma}\label{lem:nearbalE1E2}
Let $\{r,r'\}$ be a pair of rows in $F_1 \oplus F_2$ such that
$A(r,r')$ has type $\alpha$.  Then $A(r,r')$ is $(0,1)$-balanceable,
$(1,0)$-balanceable and either $(1,1)$-balanceable or
$(1,-1)$-balanceable.
\end{lemma}

\begin{proof}
Consider the arrays 
\begin{equation}\label{e:E1E2}
\begin{array}{|c|c|c|}
\hline
2x & 0  & 1 \\
\hline
  0  &  0 & 2x+1  \\
\hline
  0 & 2x & 0\\
\hline
\end{array}
\quad \text{and}\quad
\begin{array}{|c|c|c|}
\hline
2x+1 & 0 & 0 \\
\hline
  0  & 1 &  2x\\
\hline
  0 & 2x & 0\\
\hline
\end{array}
\end{equation}
of types $E_1$ and $E_2$, respectively.
By \lref{lem:balanz}, the following matrices $(p,q)$-balance
the arrays above
for $(p,q) = (0,-1)$, $(-1,0)$ and $(1,1)$, respectively:
\[
\begin{array}{|c|c|c|}
\hline
x & 0  & 0 \\
\hline
  0  &  0 & x  \\
\hline
  0 & x+1 & 0\\
\hline
%\multicolumn{3}{c}{}  \\[-13pt]
%\multicolumn{3}{c}{\text{(0,-1)-balance}}  
\end{array}
\quad \text{and}\quad
\begin{array}{|c|c|c|}
\hline
x& 0 & 0 \\
\hline
  0  & 0 &  x+1\\
\hline
  0 & x & 0\\
\hline
%\multicolumn{3}{c}{}  \\[-13pt]
%\multicolumn{3}{c}{\text{(-1,0)-balance}}  
\end{array}
\quad \text{and}\quad
\begin{array}{|c|c|c|}
\hline
x+1 & 0  & 0 \\
\hline
  0  & 0  & x  \\
\hline
  0 & x & 0\\
\hline
%\multicolumn{3}{c}{}  \\[-13pt]
%\multicolumn{3}{c}{\text{(1,1)-balance}}  
\end{array}\;.
\]
As an array is $(p,q)$-balanceable if and only if it is $(-p,-q)$-balanceable and
$A=A(r,r')$ is equivalent to one of the arrays in \eref{e:E1E2}, it follows that $A$ is both
$(0,1)$-balanceable and $(1,0)$-balanceable and either $(1,1)$-balanceable or
$(1,-1)$-balanceable.
\end{proof}

\begin{corollary}
\label{cor:setwise} 
Let $r_1$, $r_2$, $r_3$, $r_4$, $r_5$ and $r_6$ be distinct rows. 
If $A(r_1,r_2)$ is of type $\alpha$ and $A(r_3,r_4)$ is of type $\alpha$ or $\gamma$, then 
$\{r_1,r_2,r_3,r_4\}$ is balanceable. 
Moreover, if $A(r_1,r_2)$, $A(r_3,r_4)$  and $A(r_5,r_6)$ 
are each of type $\alpha$, then the set 
$\{r_1,r_2,r_3,r_4,r_5,r_6\}$ is balanceable. 
\end{corollary}

\begin{proof}
The first result follows directly from the previous two lemmas and
\lref{lem:balsets}.  For the second result, by \lref{lem:nearbalE1E2},
$A(r_1,r_2)$ can be $(1,1)$-balanced or $(1,-1)$-balanced,
$A(r_3,r_4)$ can be $(-1,0)$-balanced and $A(r_5,r_6)$ can be
$(0,1)$-balanced. The result follows from \lref{lem:balsets}.
\end{proof}

To analyse sets of rows of type $\gamma$ more closely, we consider
tables similar to \fref{figgy2}.  Let $S_i$ be the subtable of the
table in \fref{figgy2}, with rows $r$ and $T_i$.  Also, let $S^*_i$ be
a table formed from $S_i$, by replacing the row $T_i$ with a row $r'$
that has type $T_i^*$ with respect to $r$, the first row of $S_i$.  By
\lref{lem:num00gamma}, any pair of rows of type $\gamma$ include one
row of type $\gamma_1$, that is, a row with exactly $n/4$ occurrences
of $(a,b)$ for each $(a,b) \in \{0,1\}^2$.  Therefore, any pair of
rows of type $\gamma$ is equivalent to the rows of $S_i$ or an $S_i^*$
for some $i \in \{1,2,3,4\}$.

Note that by swapping $F_1$ with $F_2$ and rearranging columns, we map
$S_1$ to $S_3$ and $S_2$ to $S_4$ and vice versa.  Moreover, swapping
the symbols in $F_1$ and rearranging columns maps $S_3$ to $S_4$ and vice versa, while 
fixing $S_1$ and $S_2$. Thirdly, 
swapping
the symbols in $F_2$ and rearranging columns maps $S_1$ to $S_2$ and vice versa, while 
fixing $S_3$ and $S_4$. 
Thus, the tables $S_1$, $S_2$, $S_3$ and $S_4$ are equivalent, and
each $S_i^*$ is equivalent to some $S_1^*$.  We will exploit these
facts in the following lemmas.

Next, define $S_{i,j}$ to be the table formed by rows $r$, $T_i$ and
$T_j$ from \fref{figgy2} (if $i=j$ then we repeat row $T_i$).  Let
$S_{i,j}^*$ be a table $S_{i,j}$ with row $T_i$ replaced by a row of
type $T_i^*$ with respect to $r$ and/or with row $T_j$ replaced by a
row of type $T_j^*$ with respect to $r$.  As above, every table
$S_{i,j}$ is equivalent to either $S_{1,1}$, $S_{1,2}$ or $S_{1,3}$
and each $S_{i,j}^*$ is equivalent to $S_{1,1}^*$, $S_{1,2}^*$ or
$S_{1,3}^*$.

In \lref{lem:legswap} below, we show that any $S^*_i$ can be formed
from $S_i$ by swapping particular entries. The definition of such a
swap is in part motivated by the following observation that
will be used throughout this subsection.  By inspecting
\fref{figgy2}, one notices that if $A(r,r')$ is of type $T_3$
or $T_4$, then rows $r$ and $r'$ of $F_1$ are complementary.
Similarly, if $A(r,r')$ is of type $T_1$ or $T_2$, then rows $r$ and $r'$
of $F_2$ are complementary.  Let $r'$
be the second row of $S_1$ or $S_2$.  Then a {\em legitimate swap} in
$r'$ is a swap which replaces two cells containing 
$(a,b)$ and $(c,1-b)$ with $(a,1-b)$ and $(c,b)$, or with $(c,1-b)$ and
$(a,b)$, respectively. A legitimate swap in the second row of $S_3$ or
$S_4$ is defined analogously by interchanging the roles of $F_1$ and $F_2$.

\begin{lemma}\label{lem:legswap}
Let $i \in \{ 1,2,3,4\}$.  Then any $S^*_i$ can be
formed from $S_i$ by performing one legitimate swap in the second row.
In particular, every $S^*_i$ has exactly two columns where 
both rows have the same first entry if $i \in \{3,4\}$  
or the same second entry if $i \in \{1,2\}$.  
\end{lemma}

\begin{proof}
From the symmetries described above, it suffices to consider the case
when $i=1$.  Let $r$ and $r'$ be the first and second rows of a table
$S_1^*$, respectively.  Observe that $S^*_1$ will differ from
$S_1$ by two columns; these correspond to the $1$'s in the 3rd column
of $A(r,r')$. As every $S_i$ and $S_i^*$ have the same first row,
$S^*_1$ in fact only differs from $S_1$ by two cells in the second
row.  Let $(a,b)$ and $(c,d)$ be the entries in the two cells of $S_1$
that differ from those in $S^*_1$.  In any column of $S^*_1$
corresponding to a $1$ in the third column of $A(r,r')$, the ordered
pairs in the two rows of $S^*_1$ must have the same second entry.  On
the other hand, in every column of $S_1$ the ordered pairs in both
rows have a different second entry. So $(a,b)$ and $(c,d)$ in $S_1$
are replaced with $(a',1-b)$ and $(c',1-d)$, respectively, for some
$a',c'$.  The number of $0$'s in row $r'$ of each of $F_1$ and $F_2$
is $n/2$ only if we have the multiset equalities $\{a',c'\} = \{a,c\}$
and $\{1-b,1-d\}=\{b,d\}$, respectively.  It quickly follows that
replacing $(a,b)$ and $(c,d)$ with $(a',1-b)$ and $(c',1-d)$,
respectively, must be a legitimate swap.
\end{proof}
 
It is an immediate corollary of \lref{lem:legswap} that any
$S_{i,j}^*$ can be formed from $S_{i,j}$ by performing a legitimate
swap in the second row and/or by performing a legitimate swap in the
third row.  We now consider two rows $r_1$ and $r_2$ that do not
balance with a given row $r$ of type $\gamma_1$.  By the comments
above, any three such rows $r,r_1$ and $r_2$ are equivalent to the
rows of $S_{1,j}$ or an $S_{1,j}^*$ for $j\in \{1,2,3\}$.
Next, we focus on these situations.

Here and for the remainder of the section, we use the following definition.
Let $C$ be a subset of the columns of $F_1\oplus F_2$. 
Then let $A_C(r_1,r_2)$ be defined analogously to $A(r_1,r_2)$, 
where each cell $(i,j)$ of $A_C(r_1,r_2)$ only counts columns in $C$.

\begin{lemma}\label{lem:unbalpairsconfig}
Let $r$, $r_1$ and $r_2$ be the three rows of $S_{1,j}$ or an $S_{1,j}^*$
for $j \in \{1,2,3\}$.
Then the matrix $A(r_1,r_2)$ is of the form $A+B$, where $B$ is an
admissible matrix whose entry sum is $8$ and $A$ is the following
configuration,
\[
\begin{array}{|c|c|c|}
\hline
 0  & 0 &  0  \\
\hline
 0  &0  & 0 \\
\hline
 0 & 0 & 8y-4 \\
\hline
\end{array} \;,
\quad\quad
\begin{array}{|c|c|c|}
\hline
 0  & 0 & 4y-2  \\
\hline
 0  &0  & 4y-2 \\
\hline
 0 & 0 & 0 \\
\hline
\end{array} \;,
\quad\quad
\begin{array}{|c|c|c|}
\hline
 0  & 0 &  2y-1  \\
\hline
 2y-1  &0  & 0 \\
\hline
 0 & 2y-1 & 2y-1 \\
\hline
\end{array} 
\] 
when $j=1$, $j=2$ and $j=3$, respectively.
\end{lemma}

\begin{proof}
Let $S$ be $S_{1,j}$ or an $S_{1,j}^*$ for $j \in \{1,2,3\}$ with rows $r,r_1,r_2$.
Choose a set $C$ of 8
columns of $S$ such that exactly two columns contain the pair $(a,b)$
in row $r$ for each $(a,b)$ and every column involved in the
legitimate swaps in $r_1$ and $r_2$ (if any) are in $C$.  Note that by
\lref{lem:legswap}, such a set of columns exists, as a legitimate swap
cannot change two cells containing the same ordered pair.  Let
$C'=N(n)\setminus C$.  Then by considering \fref{figgy2},
it is easy to check that $A=A_{C'}(r_1,r_2)$ is the configuration
given in the lemma, when $j=1$, $j=2$ and $j=3$, respectively.
Finally, $B=A_{C}(r_1,r_2)$ clearly has entry sum $8$, and is
admissible, since $A(r_1,r_2)$ and $A$ are admissible. As $A(r_1,r_2)
= A_{C}(r_1,r_2)+A_{C'}(r_1,r_2)$, the result follows.
\end{proof}

We can now show a stronger analogue of Lemmas~\ref{lem:4alpharows}
and~\ref{lem:4betarows} for rows of type $\gamma$.

\begin{lemma}\label{lem:3gammarows}
In any set of three rows of type $\gamma$ in $F_1 \oplus F_2$
at least one pair of the rows is balanceable.
\end{lemma}

\begin{proof}
Let $r,r_1,r_2$ be three rows of type $\gamma$ in $F_1\oplus F_2$.
By Lemmas~\ref{lem:num00gamma} and \ref{l:not2difftype}, without loss of generality, we can assume
that $r$ is of type $\gamma_1$ and neither $\{r,r_1\}$ nor
$\{r,r_2\}$ is balanceable.  Also by equivalence, we can assume that
$r,r_1$ and $r_2$ are the rows of $S_{1,j}$ or an $S_{1,j}^*$ 
for $j \in \{1,2,3 \}$.  As $n > 4$, a row of
type $\gamma$ exists in $F_1 \oplus F_2$ only if $n \geq 12$.
By comparing the configurations $A$ in \lref{lem:unbalpairsconfig}
with the exceptional configurations in \lref{lem:exceptions},
$\{r_1,r_2\}$ is balanceable unless $n=12$, $j=3$ and
$A(r_1,r_2)=A+B$ where $A$ is the last configuration in
\lref{lem:unbalpairsconfig}, so assume these conditions hold.
The only exceptional configurations in
\lref{lem:exceptions} consistent with the last configuration of
\lref{lem:unbalpairsconfig} are $E_4$ and $E_6$.  However, $E_4$ can only
occur if $n \equiv 4 \pmod 6$, so we can assume that $A(r_1,r_2)$ is
equivalent to $E_6$.  

Let $S$ be $S_{1,3}$ or the $S_{1,3}^*$ with rows $r,r_1$ and $r_2$.
Then $S$ can be formed from 
$S_{1,3}$ by performing at most one
legitimate swap in the second row and at most one legitimate swap in
the third row.  Thus, given that the last two rows of $S_{1,3}$, that is, the rows $T_1$ and $T_3$ of \fref{figgy2},
correspond to the configuration
\[
\begin{array}{|c|c|c|}
\hline
 0  & 0 &  3  \\
\hline
 3  &0  & 0 \\
\hline
 0 & 3 & 3 \\
\hline
% \multicolumn{3}{c}{T_i,T_i} 
\end{array}
\]
and $A(r_1,r_2)$ is equivalent to $E_6$,
$A(r_1,r_2)$ must be either
\begin{equation}\label{e:n=12-3rowsproof}
\begin{array}{|c|c|c|}
\hline
 2  & 0 &  1  \\
\hline
3  &0  & 0 \\
\hline
 0 & 5 & 1 \\
\hline
% \multicolumn{3}{c}{T_i,T_i} 
\end{array}
\quad \text{or} \quad 
\begin{array}{|c|c|c|}
\hline
 0  & 0 &  5  \\
\hline
 3  &2  & 0 \\
\hline
 0 & 1 & 1 \\
\hline
\end{array}\,.
\end{equation}
Now, as $S_{1,3}$ has $n/4 =3$ columns where the second and third rows both
contain the pair $(1,1)$, $S$ must differ from $S_{1,3}$ by exactly
one legitimate swap in each of the last two rows. Moreover, the two
legitimate swaps must include distinct columns that contain $(1,1)$.
The legitimate swap in row $r_1$ replaces a $(1,1)$ in some column $c$
with $(1,0)$ or $(0,0)$.  As the legitimate swap in row $r_2$ cannot
occur in column $c$, replacing a $(1,1)$ with a $(1,0)$ in row $r_1$
is inconsistent with the configurations in \eref{e:n=12-3rowsproof},
since the entry in cell $(3,1)$ of $A(r_1,r_2)$ would then be at least
$1$.  Therefore, the legitimate swap in row $r_1$ is one which swaps a
$(1,1)$ with a $(0,0)$.  By a similar argument, the legitimate swap in
row $r_2$ is also one which swaps a $(1,1)$ with a $(0,0)$.  Hence,
$A(r_1,r_2)$ must be
\[
\begin{array}{|c|c|c|}
\hline
 1  & 0 &  3  \\
\hline
 3  &1  & 0 \\
\hline
 0 & 3 & 1 \\
\hline
\end{array}\,,
\]
which is neither of the configurations in \eref{e:n=12-3rowsproof}.
It follows that $\{r_1,r_2\}$ is
balanceable even when $n=12$, completing the proof.
\end{proof}

\begin{lemma}\label{l:any4rows}
If $n>8$, then any set of four rows in $F_1 \oplus F_2$ contains a
balanceable pair.
\end{lemma}

\begin{proof}
  By \lref{lem:3gammarows} we may assume that our four rows are
  $r,r_1,r_2,r_3$ where $r$ does not have type $\gamma$.  Now, either
  $n\not\equiv2\pmod6$ and none of $r,r_1,r_2,r_3$ has type $\alpha$, or
  $n\not\equiv4\pmod6$ and none of $r,r_1,r_2,r_3$ has type $\beta$.  Hence,
  by \lref{lem:4alpharows} and \lref{lem:4betarows}, we may assume
  that there is $r'\in\{r,r_1,r_2,r_3\}$ such that $r'$ does not have type
  $\alpha$ or $\beta$. If $r'\ne r$ then the pair $\{r,r'\}$ is
  balanceable by \lref{lem:exceptions}, since it is not of type
  $\alpha$, $\beta$ or $\gamma$.  For the same reason, if $r'=r$ then
  $\{r,r_i\}$ is balanceable for each $i\in\{1,2,3\}$.
\end{proof}

We will also require a result about balancing particular sets of four
rows.  To prove this result, we need a refined version of
\lref{lem:unbalpairsconfig} in very special cases, as in the lemma
below.  For the lemma and the remainder of the section, we use the
following definition.  Let $r_1,r_2$ be rows of type $T^*_{i_1}$ and
$T^*_{i_2}$ with respect to $r$, respectively.  We say the legitimate
swaps in $r_1$ and $r_2$ are disjoint if the columns involved in each
swap are disjoint.

\begin{lemma}\label{lem:disjswapsTs}
Let $r$, $r_1$ and $r_2$ be the rows of an $S_{1,j}^*$ for some
$j\in\{1,2\}$.  Suppose that $r_1$ and $r_2$ both have legitimate
swaps that are disjoint.  Then the entries of $A=A(r_1,r_2)$ satisfy
\begin{itemize}
\item $a_{33} = 8y$  and $a_{13}=0= a_{23}$ if $j=1$;  
\item $a_{13}+a_{23} = 8y$, $a_{33}=0$ and $a_{13},a_{23} \geq 4y-2$ if $j=2$.
\end{itemize}
\end{lemma}

\begin{proof}
Let $S$ be the $S_{1,j}^*$ with rows
$r,r_1,r_2$ and let $r_1'$ and $r_2'$ be second and third rows of $S_{1,j}$, respectively.
By considering
\fref{figgy2}, it is easy to check that $A(r_1',r_2')$ is
\[
\begin{array}{|c|c|c|}
\hline
 0  & 0 &  0  \\
\hline
 0  &0  & 0 \\
\hline
 0 & 0 & 8y+4 \\
\hline
% \multicolumn{3}{c}{T_i,T_i} 
\end{array}
\quad \text{and} \quad
\begin{array}{|c|c|c|}
\hline
 0  & 0 & 4y+2  \\
\hline
 0  &0  & 4y+2 \\
\hline
 0 & 0 & 0 \\
\hline
% \multicolumn{3}{c}{T_1,T_4}   
 \end{array}
\]
when $j=1$, and $j=2$, respectively.  By \lref{lem:legswap}, rows
$r_1$ and $r_2$ differ from, respectively, $r_1'$ and $r_2'$ by a
legitimate swap.
By assumption, these legitimate swaps are disjoint, so
$S$ and $S_{1,j}$ differ by $4$ cells located in different columns; call
this set of $4$ columns $C$ and let $C' = N(n)\setminus C$.
Clearly, $A_{C}(r_1,r_2)$ has entry sum $4$ and $A_{C'}(r_1,r_2)$ has
entry sum $8y$ with all non-zero cells occurring in the last column.
In each column in $C$ the element in $S$ in exactly one of the rows
$r_1$ and $r_2$ has a different second entry to the element in row
$r$, by \lref{lem:legswap}.  It follows that $A_{C}(r_1,r_2)$ has
$0$'s in the last column.  As
$A=A(r_1,r_2)=A_{C}(r_1,r_2)+A_{C'}(r_1,r_2)$, the last column of
$A(r_1,r_2)$ has entry sum $8y$, which, along with the configurations
above, imply that $a_{33} = 8y$ and $a_{13}=0=a_{23}$ when $j=1$ and
$a_{13}+a_{23}=8y$ when $j=2$.  Finally, as $S$ and $S_{1,j}$ differ
in exactly $4$ cells, neither $a_{13}$ nor $a_{23}$ can be less than
$4y-2$ when $j=2$, and the result follows.
\end{proof}

Finally, we also require a result about three rows that do not balance with a given row of type $\gamma_1$.
\begin{lemma}\label{lem:balpart4rows}
Let $r$, $r_1$, $r_2$ and $r_3$ be rows of $F_1 \oplus F_2$ such that
$r$ is of type $\gamma_1$ and $\{r,r_j\}$ has type $T_{i_j}$ or
$T^*_{i_j}$ for $j\in\{1,2,3\}$, where $i_1$, $i_2$, $i_3\in \{1,2,3,4\}$. If either 
\begin{itemize} 
\item[{\rm (i)}] $\{i_1,i_2,i_3\}\cap\{1,2\}\ne\emptyset$
and $\{i_1,i_2,i_3\}\cap\{3,4\}\ne\emptyset$; or 
\item[{\rm (ii)}] $\{r,r_j\}$ has type $T^*_{i_j}$ for $j\in\{1,2,3\}$
such that the legitimate swaps in two of $r_1,r_2,r_3$ are disjoint
\end{itemize}
then $\{ r,r_1,r_2,r_3\}$ is balanceable.
\end{lemma}

\begin{proof}
By equivalence, it suffices to consider the case when $i_1=1$ and
either $i_2=3$ or $i_2,i_3 \in \{1,2\}$ and the rows satisfy (ii).  First suppose that $i_2=3$.
By \lref{lem:unbalpairsconfig}, $A(r_1,r_2)$ is
\[
\begin{array}{|c|c|c|}
\hline
 0  & 0 &  2y-1  \\
\hline
 2y-1  &0  & 0 \\
\hline
 0 & 2y-1 & 2y-1 \\
\hline
\end{array}
+B
\]
where $B$ is an admissible matrix with entry sum $8$.
As $y \geq 1$, we can write $A(r_1,r_2)$ as 
\[
\begin{array}{|c|c|c|}
\hline
 0  & 0 &  1  \\
\hline
 1  &0  & 0 \\
\hline
 0 & 1 & 1 \\
\hline
\end{array}+B'
\]
where $B'$ is some admissible matrix with entry sum $n-4=8y$.  If $B'$
is not an exceptional configuration in \lref{lem:exceptions}, then 
there exists a matrix $B''$ which $(0,0)$-balances $B'$.
The 
following then show that $A(r_1,r_2)$ is $(0,1)$-balanceable and
$(1,0)$-balanceable, respectively, by \lref{lem:balanz}:
\[
\begin{array}{|c|c|c|}
\hline
 0  & 0 &  1  \\
\hline
 1  &0  & 0 \\
\hline
 0 & 0 & 0 \\
\hline
\end{array} +B''
\quad \text{and}\quad 
\begin{array}{|c|c|c|}
\hline
 0  & 0 &  1  \\
\hline
 0  &0  & 0 \\
\hline
 0 & 0 & 1 \\
\hline
\end{array}+B''. 
\]
  As $A(r,r_3)$ is equivalent to
$E_5$ or $E_6$, \lref{lem:nearbalE5E6} implies that $A(r,r_3)$ is 
either $(0,1)$-balanceable or $(1,0)$-balanceable. Therefore,
$\{r,r_1,r_2,r_3\}$ is balanceable, by \lref{lem:balsets}.  If
$n\geq20$, then $b'_{33} \geq 2y-2 \geq 2$, so $B'$ cannot be an
exceptional configuration in \lref{lem:exceptions}, by \Cref{goodcases}.  Thus, if $B'$ is
an exceptional configuration in \lref{lem:exceptions}, then $n=12$ and
$B'$ must be equivalent to $E_1$ or $E_2$, as $B'$ has entry sum $8$.
By \lref{lem:nearbalE1E2}, there are matrices $B''$ and $B'''$ that
$(1,0)$-balance and $(0,-1)$-balance $B'$, respectively.  Thus 
\[
\begin{array}{|c|c|c|}
\hline
 0  & 0 &  0  \\
\hline
 1  &0  & 0 \\
\hline
 0 & 0 & 1 \\
\hline
\end{array}+B''
\quad \text{and} \quad 
\begin{array}{|c|c|c|}
\hline
 0  & 0 &  0  \\
\hline
 1  &0  & 0 \\
\hline
 0 & 0 & 1 \\
\hline
\end{array}+B'''
\]
$(0,1)$-balance and $(-1,0)$-balance $A(r_1,r_2)$, respectively.  As
before, \lref{lem:nearbalE5E6} implies $A(r,r_3)$ has to be either
$(0,1)$-balanceable or $(1,0)$-balanceable, so $\{r,r_1,r_2,r_3\}$ is
balanceable by \lref{lem:balsets}.

Now suppose that $i_2,i_3 \in \{1,2\}$ and condition (ii) is satisfied. Without loss of generality,
let rows $r_1$ and $r_2$ have disjoint legitimate swaps.  First
suppose that $i_2=1$. Then by \lref{lem:disjswapsTs}, $A(r_1,r_2)$ is
of the form
\[
\begin{array}{|c|c|c|}
\hline
 0  & 0 &  0  \\
\hline
 0  &0  & 0 \\
\hline
 0 & 0 & 8y \\
\hline
\end{array}+B
\]
for some admissible matrix $B$, with entry sum $4$ and $0$'s in the last column.
In particular, $B$ has at least one non-zero entry in the first column.
Thus at least one of 
\[
\begin{array}{|c|c|c|}
\hline
 1  & 0 &  0  \\
\hline
 0  &0  & 0 \\
\hline
 0 & 0 & 4y +1\\
\hline
\end{array}
\quad \text{ or } \quad 
\begin{array}{|c|c|c|}
\hline
 0  & 0 &  0  \\
\hline
 1  &0  & 0 \\
\hline
 0 & 0 & 4y+1 \\
\hline
\end{array}
\quad \text{ or } \quad 
\begin{array}{|c|c|c|}
\hline
 0  & 0 &  0  \\
\hline
 0  &0  & 0 \\
\hline
 1 & 0 & 4y+1 \\
\hline
\end{array}
\]
$(p,q)$-balances $A(r_1,r_2)$, where $(p,q)$ is $(1,1)$, $(-1,1)$ or
$(0,1)$, respectively.  As $A(r,r_3)$ is equivalent to $E_6$,
\lref{lem:nearbalE5E6} implies that $A(r,r_3)$ is $(1,1)$-balanceable,
$(-1,1)$-balanceable and $(0,1)$-balanceable.  Hence,
\lref{lem:balsets} implies that $\{r,r_1,r_2,r_3\}$ is balanceable.

Finally, suppose that $i_2=2$. Then, by \lref{lem:disjswapsTs},
$A=A(r_1,r_2)$ satisfies $a_{13}+a_{23}=8y$, $a_{33}=0$ and
$a_{13},a_{23} \geq 4y-2$.  In particular, the first two columns of
$A$ each have at least one non-zero entry.  So if $a_{13}$ and
$a_{23}$ are both at least $2y+1$, then at least one of the following
matrices $(p,q)$-balances $A$ for some $(p,q)\in \{(0,1),(1,1)\}$:
\[
\begin{array}{|c|c|c|}
\hline
 1  & 0 &  2y  \\
\hline
 0  &0  & 2y+1 \\
\hline
 0 & 0 & 0 \\
\hline
\end{array}
\quad \text{and} \quad
 \begin{array}{|c|c|c|}
\hline
 0  & 0 &  2y+1  \\
\hline
 1  &0  & 2y\\
\hline
 0 & 0 & 0 \\
\hline
\end{array}
\quad \text{and} \quad
 \begin{array}{|c|c|c|}
\hline
 0  & 0 &  2y+1  \\
\hline
 0  &0  & 2y \\
\hline
 1 & 0 & 0 \\
\hline
\end{array}\,.
\]

Otherwise, as $a_{13},a_{23} \geq 4y-2$ and $y \geq 1$, one of $a_{13}$ and
$a_{23}$ is less than $2y+1$ only if $y=1$ and $a_{13}$ or $a_{23}$ is
$4y-2$.  If $a_{13} = 4y-2$, then $a_{23}=4y+2$ and so one of $a_{11}$
and $a_{12}$ is non-zero. It follows that at least one of the
following $(0,\pm1)$-balances $A$ 
\[
 \begin{array}{|c|c|c|}
\hline
 1  & 0 &  2y  \\
\hline
 0  &0  & 2y+1\\
\hline
 0 & 0 & 0 \\
\hline
\end{array}
\quad \text{and} \quad
 \begin{array}{|c|c|c|}
\hline
 0  & 1 &  2y  \\
\hline
 0  &0  & 2y+1 \\
\hline
 0 & 0 & 0 \\
\hline
\end{array}
\]
and so $A$ is $(0,1)$-balanceable.  Similarly, if $a_{23}=4y-2$, then
$A$ is $(0,1)$-balanceable.  In all cases $A$ is 
$(0,1)$-balanceable or $(1,1)$-balanceable.  As $A(r,r_3)$ has type
$T^*_{i_3}$, it is $(0,1)$-balanceable and $(1,1)$-balanceable,
by \lref{lem:nearbalE5E6}. Hence, $\{r,r_1,r_2,r_3\}$ is balanceable,
by \lref{lem:balsets}.
\end{proof}

We end this subsection with two results that limit the number of rows
that $F_1 \oplus F_2$ can have of a particular type. In both we will need
to use that $F_1$ is orthogonal to $F_2$, an assumption that we have not
needed until now.

\begin{lemma}\label{dubya}
Let $F_1$ and $F_2$ be orthogonal and $n>8$. Also let $w_{\alpha}(i)$
be the number of type $\alpha_i$ rows and $w_{\beta}(i)$ the number of
type $\beta$ rows in $F_1 \oplus F_2$ for each $i\in \{1,2\}$. Then
$$\max\{w_{\alpha}(1),
w_{\alpha}(2),w_{\beta}(1),w_{\beta}(2)\} \leq \frac{3n^2}{4(n-2)}< 3n/4+2.$$
\end{lemma}

\begin{proof}
Suppose that $w_{\alpha}(1)>\frac{3n^2}{4(n-2)}$. Then by
\lref{lem:num00alpha}, there are more than
$\frac{3n^2}{4(n-2)}\times\frac{n-2}{3}= \frac{n^2}{4}$ occurrences of
$(0,0)$ in $F_1 \oplus F_2$, contradicting the assumption that $F_1$
and $F_2$ are orthogonal.  The proofs for the other cases are similar.
\end{proof}

\begin{lemma}\label{lem:numofsimtype}
Let $F_1$ and $F_2$ be orthogonal and $r$ be a row of type $\gamma_1$
in $F_1\oplus F_2$. Then there are at most $\frac{n}{2}+1$ rows $r'$
such that $A(r,r')$ has a type from the set $\{T_1, T^*_1, T_2,
T^*_2\}$.  Furthermore, if there are exactly $\frac{n}{2}+1$ such
rows, then at least $\frac{n}{2}$ of them must be types $T^*_1$ or
$T^*_2$ and every column has at least one row with a legitimate swap
in that column. In particular, there are at least two rows with
legitimate swaps in disjoint pairs of columns. All of the above statements
hold with $T_1, T^*_1, T_2, T^*_2$ replaced respectively by
$T_3, T^*_3, T_4, T^*_4$.
\end{lemma}

\begin{proof}
We only prove the claim about $T_1, T^*_1, T_2, T^*_2$, as the other case is
equivalent.  Let $S$ be the submatrix of $F_1 \oplus F_2$ with $1+t$
rows consisting of $r$ and all other rows that are of types
$T_1,T_2,T^*_1$ or $T^*_2$ with respect to $r$.
Let $S'$ be the $(t+1)\times n$
matrix formed from $S$ by replacing each row $r'$ of type $T^*_1$ or
$T^*_2$ with respect to $r$ with the row $r''$ of type $T_1$ or $T_2$
with respect to $r$, respectively. For each $(a,b) \in \{0,1\}^2$, let
$C_{ab}$ be the submatrix of $S$ with all the rows of $S$, and the
columns for which $r$ contains $(a,b)$.
Define $C'_{ab}$ from $S'$ in the analogous way, for each $(a,b)\in\{0,1\}^2$.

By \lref{lem:legswap}, each row of $S'$ (except the first row)
can be formed from the
corresponding row of $S$ by performing a single legitimate swap.
Let $t_{ab}$ be the number of elements in $C'_{ab}$ that differ from
those in $C_{ab}$. We then have that $t_{00}+t_{01}+t_{10}+t_{11} \leq 2t$.
Let $t_{\min}= \min\{t_{00},t_{01},t_{10},t_{11}\}$. 
By the pigeonhole principle, $t_{\min} \leq t/2$. 
Recall that if $A(r,r')$ is of type $T_1$ or $T_2$, then every column
has a different entry in rows $r$ and $r'$ of $F_2$.  So, the columns
of $C_{a(1-b)}$ contain at least $tn/4-t_{a(1-b)}$ occurrences of
the symbol $b$ in $F_2$.  As each $C_{ab}$ is a subset of $n/4$
columns of $F_1\oplus F_2$, we must have that $tn/4-t_{ab} \leq n^2/8$, for
each $ab$.  Therefore,
\[
n^2/8 \geq \frac{tn}{4}-t_{\min} \geq \frac{tn}{4}-\frac{t}{2}
\]
and rearranging for $t$ implies that
$t\leq\frac{n^2}{2(n-2)}=\frac{n+2}{2}+\frac{2}{n-2}$.
Since we are assuming that $n\ge8$ and $t$ is an integer, it follows that
$t\le(n+2)/2$,
proving the first claim of the lemma.

Finally, suppose that $t = ({n+2})/{2}$, and hence
$t_{ab} \geq tn/4-n^2/8=n/4$ for each $ab$.
So there are at least $n$ elements in
$S'$ that differ from the corresponding element in $S$.  As each row of
$S'$ differs from the corresponding row in $S$ in 0 or 2 places, there
are at least $n/2$ rows in which $S$ and $S'$ differ.  Also, if there were
a column in which $S$ and $S'$ agreed, then the corresponding column
of $F_2$ would contain at least $({n+2})/{2}$ copies of some symbol,
violating the fact that $F_2$ is a frequency square. Hence there are either
$n/2$ or $n/2+1$ legitimate swaps which between them cover all $n$ columns.
It follows that at least two of them must involve disjoint pairs of columns.
\end{proof}

\subsection{Proof of \tref{t:main}}\label{subsec:proofofthm}

We separate the proof of \tref{t:main} into cases depending on the type of
rows present.  We begin with the case when there is a row not of type
$\gamma$.  Here and for the remainder of the paper, we assume that
$F_1$ and $F_2$ are \emph{orthogonal} frequency square of order $n\geq 8$.

\begin{lemma}\label{lem:nogamcanbalrows} 
Let $n\notin\{8,20\}$ and suppose there exists a row which is not of type
$\gamma$ in $F_1 \oplus F_2$.  Then the rows of $F_1\oplus F_2$ can be
partitioned into balanceable sets. 
\end{lemma}

\begin{proof}
By \lref{l:not2difftype}, no row of $F_1\oplus F_2$ can have
two different types from the set 
$\{\alpha_1,\alpha_2,\beta_1,\beta_2,\gamma_1,\gamma_2\}$. 
In particular, any two rows not of the same type
always form a balanceable pair, unless they are of types $\gamma_1$ and $\gamma_2$.   
By \lref{l:any4rows}, we can partition the rows of $F_1 \oplus F_2$
into pairs ${\mathcal R}$ such that at most one pair is not
balanceable, by greedily selecting pairs of rows that balance.  We are
done unless there is a pair $\{r,r'\} \in {\mathcal R}$ that is not
balanceable.  We proceed by showing there is always a way to re-pair
$r$ and $r'$ so that all pairs in ${\mathcal R}$ are balanceable.

Suppose first that $\{r,r'\}$ is of type $\alpha$ or $\beta$.
We assume that the pair is of type $\alpha_1$; the other cases are similar.
Since $n\geq 12$, \lref{dubya} implies that there
exists at least two rows $v$ and $w$ which are not of type
$\alpha_1$. If $\{v,w\}\in {\mathcal R}$, then we can replace $\{r,r'\}$ and
$\{v,w\}$ with $\{r,v\}$ and $\{r',w\}$, both of which must be balanceable.
Otherwise, $\{v,v'\},\{w,w'\}\in {\mathcal R}$ for some rows $v'$ and $w'$.  If
$\{r,v'\}$ is balanceable, we are done, as we can replace pairs
$\{r,r'\}$ and $\{v,v'\}$ in ${\mathcal R}$ with the balanceable pairs
$\{r,v'\}$ and $\{r',v\}$.  We are similarly done if any of $\{r',v'\}$,
$\{r,w'\}$ or $\{r',w'\}$ are balanceable.  By \lref{l:any4rows},
at least one pair from $\{r,r',v',w'\}$ is balanceable.  So we
are done unless $\{v',w'\}$ is balanceable, in which case we can
replace pairs $\{r,r'\}$, $\{v,v'\}$ and $\{w,w'\}$ in ${\mathcal R}$
with balanceable pairs $\{r,v\}$, $\{r',w\}$ and $\{v',w'\}$.

Finally suppose that rows $r$ and $r'$ are of type $\gamma$.  By assumption,
there is a row $v$ that is not of type~$\gamma$.  Then both $\{r,v\}$ and
$\{r',v\}$ are balanceable.  Let $v'$ be the row such that
$\{v,v'\}\in {\mathcal R}$. If $v'$ is not of type $\gamma$, then
$\{r,v'\}$ and $\{r',v'\}$ are both balanceable. If $v'$ is of type
$\gamma$, then at least one of $\{r,v'\}$ and $\{r',v'\}$ is
balanceable, by \lref{lem:3gammarows}.  In any case, the pairs
$\{r,r'\}$ and $\{v,v'\}$ in ${\mathcal R}$ can be replaced with two
balanceable pairs.  This completes the proof.
\end{proof}

The cases when $n=8$ or $n=20$ are dealt with separately in the
following two lemmas.

\begin{lemma}\label{lem:n=8somealpharow}
If $n=8$, then the rows of $F_1\oplus F_2$ can be partitioned into balanceable
sets.
\end{lemma}

\begin{proof}
Observe that $F_1\oplus F_2$ has rows of neither type $\beta$ nor $\gamma$, 
since $n=8$. 

Let $G$ be the graph whose vertices are the rows of $F_1\oplus F_2$
with an edge between rows $v$ and $v'$ if and only if $\{v,v'\}$ is
balanceable. We may assume that $G$ has no perfect matching since otherwise
the rows of $F_1\oplus F_2$ can be partitioned into balanceable pairs.

Let $\mathcal{R}$ be a partition of the rows of $F_1\oplus F_2$ into
pairs, with as few balanceable pairs as possible.
\Cref{cor:setwise} implies that, if $\mathcal{R}$ contains more than
one unbalanceable pair, then the rows of $F_1\oplus F_2$ can be partitioned
into balanceable sets.  Thus, we may assume that $\mathcal{R}$ contains
a pair of rows $\{r,r'\}$ of type $\alpha$, and the rows other than
$r$ and $r'$ induce a clique in $G$.  If there are two
disjoint edges in $G$ incident to $r$ and $r'$, then there exists a
perfect matching in $G$.  Alternatively, if there exists two disjoint
pairs of rows both of which are not balanceable, then it violates our choice
of $\mathcal{R}$.
It follows that $G$ must be the disjoint union of a $K_7$ and
$K_1$, where the isolated vertex is either $r$ or $r'$.

Finally, we show that $G$ cannot be the disjoint union of $K_7$ and
$K_1$. Assume otherwise and let $r$ be the isolated vertex.  That
is, assume that $\{r,v\}$ is not balanceable for all rows $v \neq r$ of
$F_1 \oplus F_2$.  Without loss of generality, we can assume that $r$
has type $\alpha_1$.  Then $\zz(r)\in\{2,3\}$, by
\lref{lem:num00alpha}.  If $\zz(r)=2$, then the remaining rows must
have $1$ or $3$ occurrences of $(0,0)$ each, by
\lref{lem:num00alpha}. However, we then have the contradiction that
the total number of occurrences of $(0,0)$ in $F_1\oplus F_2$ is odd.
Lastly, if $\zz(r)=3$, then the remaining rows each have at least $2$
occurrences of $(0,0)$, by \lref{lem:num00alpha}. However, this would
mean that $F_1\oplus F_2$ has at least $3+7\times 2 = 17$ occurrences
of $(0,0)$, contradicting the fact that $F_1$ and $F_2$ are orthogonal.
This completes the proof.
\end{proof}

\begin{lemma}\label{lem:n=20somealpharow}
Let $n=20$. If there exists a row of type $\alpha$ in $F_1\oplus F_2$,
then the rows of $F_1\oplus F_2$ can be partitioned into balanceable
sets.
\end{lemma}

\begin{proof}
Let $\{r,r'\}$ be a pair of rows of type $\alpha$.  As $n=20\equiv2\pmod 6$,
every row that is in a pair that is not balanceable is of type $\alpha$ or
$\gamma$.  Within any four rows of $F_1\oplus F_2$, at least one pair
is balanceable by \lref{l:any4rows}. So, we can partition the rows of
$F_1 \oplus F_2$ into pairs $\mathcal{R}$, such that $\{r,r'\} \in
\mathcal{R}$ and at most one pair other than $\{r,r'\}$ is not
balanceable.  If $\mathcal{R}$ contains a pair $\{v,v'\}$ distinct
from $\{r,r'\}$ that is not balanceable, then $\{v,v'\}$ must be of
type $\alpha$ or $\gamma$.  Therefore by \Cref{cor:setwise}, 
$\{r,r',v,v'\}$ is balanceable and it
follows that the rows of $F_1\oplus F_2$ can be partitioned into
balanceable sets. 
So suppose that $\{r,r'\}$ is the only pair in $\mathcal{R}$ that is
not balanceable.  We show that there is always a way to re-pair the
pairs in $\mathcal{R}$ so that all pairs are balanceable.  Without
loss of generality, we can assume that $\{r,r'\}$ has type $\alpha_1$
and that $\zz(r)\in\{6,7\}$ and $\zz(r')=7$ by \lref{lem:num00alpha}.

The average value for $\zz(\cdot)$ across all rows is 5, so
by the pigeonhole principle there exists distinct pairs
$\{v,v'\},\{w,w'\}\in\mathcal{R}$ such that $\zz(v')\le4$ and
$\zz(w')\le5$.
It follows from \lref{lem:num00alpha} and \lref{lem:num00gamma}
that $\{r,v'\}$, $\{r',v'\}$ and
$\{r',w'\}$ are all balanceable pairs.
If either $\{r,v\}$ or $\{r',v\}$ is balanceable then we can replace
$\{r,r'\}$ and $\{v,v'\}$ by two balanceable pairs and we are done.
So assume that is not the case. It then follows from
\lref{lem:num00alpha} and \lref{lem:num00gamma} that
$\zz(r)=7$ or $\zz(v)=7$. By interchanging $r$ and $v$ if necessary,
we may assume that $\zz(r)=7$. It then follows that $\{r,w'\}$ is balanceable.

Finally, we apply \lref{l:any4rows} to find that there must
be a balanceable pair among $\{r,r',v,w\}$. This pair, together with
two of the balanceable pairs 
$\{r,v'\}$, $\{r,w'\}$, $\{r',v'\}$, $\{r',w'\}$, $\{v,v'\}$ and
$\{w,w'\}$, can be used to replace the pairs $\{r,r'\}$,
$\{v,v'\}$ and $\{w,w'\}$.
\end{proof}

Next, we consider the case where every pair of rows is either
balanceable or of type $\gamma$.  To do so we require the following
simple graph theoretical result.

\begin{lemma}\label{lem:graphwithprop}
Let $G$ be a simple graph with an even number of vertices such that
each subset of three vertices induces at least one edge. Then either
$G$ has a perfect matching or $G$ is the disjoint union of
two odd cliques.
\end{lemma}

\begin{proof}
  Suppose that $G$ has no perfect matching. Then by Tutte's criterion there
  exists a set $S$ of vertices whose removal leaves at least $|S|+1$ components
  of odd order. But the given condition means that no induced subgraph of $G$
  has more than $2$ components. Given that
  $G$ has an even number of vertices, the only possibility is that
  $S=\emptyset$ and that $G$ has two components, both of odd order.
  Considering each set of 3 vertices from
  2 different components then shows that each component is a clique. 
\end{proof}

\begin{lemma}\label{lem:onlygammarows}
If every pair of rows of $F_1\oplus F_2$ is either balanceable or of
type $\gamma$, then the rows of $F_1\oplus F_2$ can be partitioned
into balanceable sets.
\end{lemma}

\begin{proof}
There is nothing to prove unless some rows have type $\gamma$, so we
may assume that $ n \equiv 4 \pmod 8$.
Let $G$ be a graph with rows in $F_1\oplus F_2$ as vertices and an
edge between $r$ and $r'$ if and only if $\{r,r'\}$ is balanceable.
A perfect matching in $G$ corresponds to a partition of the
rows of $F_1\oplus F_2$ into balanceable pairs. So, by
Lemmas~\ref{lem:3gammarows} and~\ref{lem:graphwithprop}, we are done
unless $G$ is the union of two disjoint odd cliques.  So, without loss
of generality, let $K_a$ and $K_{n-a}$ be the connected components of
$G$, with $a < n-a$ and $a$ odd.  Suppose there exists $4$ rows
$r,r_1,r_2,r_3$, with $r$ in $K_a$ and $r_1,r_2,r_3 $ in $K_{n-a}$,
such that $\{r,r_1,r_2,r_3\}$ is balanceable. Then the induced
subgraph of $G$ on the remaining rows forms two disjoint even cliques
and so the remaining rows can be partitioned into balanceable pairs.
It then follows that the rows of $F_1\oplus F_2$ can be
partitioned into balanceable sets.

So it suffices to show that such a set of four rows exists. Note that
as $G$ is the disjoint union of two cliques, every row is in some pair
that is not balanceable. So, by assumption, every row is of type $\gamma$.
We claim that there is at least one row in $K_a$ that is of
type $\gamma_1$. If there were no such row, then every row in $K_a$ is
of type $\gamma_2$. Inspecting \lref{lem:num00gamma}, we see that
any two $\gamma_2$ rows
form a balanceable pair, so $K_a$ contains every $\gamma_2$ row.
Also, every $\gamma_2$ row contains an odd number of occurrences of $(0,0)$,
while any $\gamma_1$ row contains an even number of occurrences of
$(0,0)$.  In total, there are $n^2/4 \equiv 0 \pmod{2}$ occurrences of
$(0,0)$ in the rows of $F_1\oplus F_2$.  We conclude
that $K_a$ contains a row $r$ of type $\gamma_1$.

If there are rows $r_1,r_2$ and $r_3$ in $K_{n-a}$ such that
$A(r,r_1)$ and $A(r,r_2)$ are of types $T_i$ or $T^*_i$ and $T_j$ or
$T^*_j$ with $i\in\{1,2\}$ and $j \in \{3,4\}$, then by
\lref{lem:balpart4rows},
$\{r,r_1,r_2,r_3\}$ is balanceable.  So we can assume that, without
loss of generality, every row in $K_{n-a}$ is of type $T_1,T^*_1,T_2$
or $T^*_2$, with respect to $r$.  By \lref{lem:numofsimtype}, if
follows that $n-a=n/2+1$ and there are at least two rows $r_1$ and
$r_2$, such that $r_1$ and $r_2$ are of types $T^*_1$ or $T^*_2$ with
respect to $r$ and $r_1$ and $r_2$ have disjoint legitimate
swaps. Choose any $r_3$ in $K_{n-a}$ distinct from $r_1$ and $r_2$
that is of type $T^*_1$ or $T^*_2$ with respect to $r$; such a row
exists since at least $n/2\geq 6$ rows in $K_{n-a}$ are type $T^*_1$
or $T^*_2$ with respect to $r$, by \lref{lem:numofsimtype}. Then by
\lref{lem:balpart4rows}, $\{r,r_1,r_2,r_3\}$ is balanceable. This
completes the proof.
\end{proof}

We can now prove \tref{t:main}.

\begin{proof}[Proof of \tref{t:main}]
Let $R$ be the set of rows of $F_1\oplus F_2$.  By
\lref{lem:partofbalsetssuff}, it suffices to show that $R$ can be
partitioned into balanceable sets.  If all non-balanceable pairs from
$R$ are of type $\gamma$ then we are done, by
\lref{lem:onlygammarows}. So, assume that $\{r,r'\}\subset R$ is a
non-balanceable pair not of type $\gamma$, from which it follows,
without loss of generality, that $r$ is not of type $\gamma$.  If
$n\notin\{8,20\}$ then \lref{lem:nogamcanbalrows} implies that we can
partition $R$ into balanceable sets. Meanwhile, if $n=8$, then we can
partition $R$ into balanceable sets, by \lref{lem:n=8somealpharow}.
Finally, if $n=20$, then $n\equiv2\pmod 6$, so $r$ must be of type
$\alpha$, by Lemma~\ref{lem:exceptions}. Hence,
\lref{lem:n=20somealpharow} completes the proof.
\end{proof}

It may be possible to prove the analogue of \tref{t:main} for
$n\equiv2\pmod4$ by similar methods. However, new configurations arise
in (the analogue of) \lref{lem:exceptions}, making the subsequent
analysis substantially more complicated.  It was important for our
proof that only certain rows can be in non-balanceable pairs (as shown
by Lemmas \ref{lem:num00alpha}, \ref{lem:num00beta} and
\ref{lem:num00gamma}). However, if we assume that $n\equiv2\pmod4$,
then \emph{any} pair of rows in $F_1 \oplus F_2$ that are
complementary in at least one of $F_1$ and $F_2$ is not balanceable.
Let $r$ be a row in $F_1 \oplus F_2$ and let $r'$ be the row which
agrees with $r$ in $F_1$ and is complementary to $r$ in $F_2$.  Let
$s$ and $s'$ be the complementary rows to $r$ and $r'$,
respectively. Then no pair of $\{r,r',s,s'\}$ is balanceable.  This
means that the analogue of \lref{l:any4rows} fails for
$n\equiv2\pmod4$.

\subsection*{Acknowledgements}
This work was supported in part by Australian Research Council grant
DP150100506.

  \let\oldthebibliography=\thebibliography
  \let\endoldthebibliography=\endthebibliography
  \renewenvironment{thebibliography}[1]{%
    \begin{oldthebibliography}{#1}%
      \setlength{\parskip}{0.4ex plus 0.1ex minus 0.1ex}%
      \setlength{\itemsep}{0.4ex plus 0.1ex minus 0.1ex}%
  }%
  {%
    \end{oldthebibliography}%
  }

\end{document}